\documentclass[11pt]{article}
\usepackage[utf8]{inputenc}
\usepackage{amssymb,graphicx,mathtools,amsfonts}
\usepackage{amsmath}
\usepackage{amsthm}
\usepackage{hyperref}
\usepackage[dvipsnames]{xcolor}
\usepackage{dsfont}

\usepackage{a4wide}
\usepackage{esint}
\catcode`\@=11
\def\theequation{\@arabic{\c@section}.\@arabic{\c@equation}}
\catcode`\@=12
\usepackage{enumitem}
\graphicspath{{images/}}
\newtheorem{theorem}{Theorem}[section] 
\newtheorem{lemma}{Lemma}[section] 
\newtheorem{definition}{Definition}[section] 

\newcommand{\R}{\mathbb{R}}

\newcommand{\al} {\alpha}

\newcommand{\de} {\delta}

\newcommand{\om} {\Omega}

\newcommand{\la} {\lambda}
\newcommand{\La} {\Lambda}
\newcommand{\noi} {\noindent}

\newcommand{\real}{\mathbb{R}}

\newcommand{\rnn}{\mathbb{R}^{N}}
\newcommand{\rtwon}{\mathbb{R}^{2N}}
\newcommand{\lv}{\lVert}
\newcommand{\rv}{\rVert}

\newcommand{\grad}{\nabla}
\newcommand{\ntrl}{\mathbb{N}}
\newcommand{\ka}{\kappa}

\newcommand{\X}{\mathcal{X}_0} 
\newcommand{\Xe}{\mathcal{X}_{0},\varepsilon} 
\newcommand{\J}{\mathcal{J}_{\mu,\lambda}} 
\newcommand{\Je}{\mathcal{J}_{\mu,\lambda}^{\varepsilon}} 
\newcommand{\Ie}{\mathcal{I}_{\mu,\lambda}^{\varepsilon}} 
\newcommand{\Jplus}{\mathcal{J}_{\mu,\lambda}^{+}} 

\newcommand{\Jep}{\mathcal{J}_{\mu, \lambda}^{\varepsilon,+}}
\newcommand{\Jepo}{\mathcal{J}_{\mu, \lambda_0}^{\varepsilon,+}}
\newcommand{\Q}{\mathcal{Q}_{\la,\mu}}

\allowdisplaybreaks

\title{Asymptotic behaviour and existence of positive solutions for mixed local nonlocal elliptic equations with Hardy potential}
\author{Shammi Malhotra\footnote{Department of Mathematics, Indian Institute of Technology Delhi, Hauz Khas New Delhi 110016,  India, shammi22malhotra@gmail.com}, Sarika Goyal\footnote{Department of Mathematics, Netaji Subhas University of Technology,
		Dwarka Sector-3, Dwarka, Delhi, 110078, India, sarika1.iitd@gmail.com, sarika@nsut.ac.in},  and K. Sreenadh\footnote{Department of Mathematics, Indian Institute of Technology Delhi, Hauz Khas New Delhi 110016,  India, sreenadh@maths.iitd.ac.in}}
\date{}

\begin{document}
\maketitle

\begin{abstract}
\noindent We investigate the existence and multiplicity of positive solutions to the following problem driven by the superposition of the Laplacian and the fractional Laplacian with Hardy potential
\begin{equation*}
\left\{
\begin{aligned}
    -\Delta u + (-\Delta)^s u - \mu \frac{u}{|x|^2} &= \lambda |u|^{p-2} u + |u|^{2^*-2} u \quad \text{in } \Omega \subset \mathbb{R}^N, \\
    u &= 0 \quad \text{in } \mathbb{R}^N \setminus \Omega,
\end{aligned}
\right.
\end{equation*}
where \( \Omega \subset \mathbb{R}^N \) is a bounded domain with smooth boundary, \( 0 < s < 1 \), \( 1 < p < 2^* \), with \( 2^* = \frac{2N}{N-2} \), \( \lambda > 0 \), and \( \mu \in (0, \bar{\mu}) \) where $\bar \mu = \left( \frac{N-2}{2} \right)^2$. 

\noindent The aim of this paper is twofold. First, we establish uniform asymptotic estimates for solutions of the problem by means of a suitable transformation. Then, according to the value of the exponent \(p\), we analyze three distinct cases and prove the existence of a positive solution. Moreover, in the sublinear regime \(1 < p < 2\), we demonstrate the existence of multiple positive solutions for small perturbations of the fractional Laplacian.\\

\medskip
\noi\textbf{Keywords:} Mixed local nonlocal operator; Hardy potential; critical exponent; multiplicity of positive solutions; asymptotic estimates.  \\
\medskip
\noi\textbf{Mathematics Subject Classification:}  35A21, 35B09, 35B33, 35J20, 35M12.
\end{abstract}

\section{Introduction}
In this paper, we are concerned with the following problem
\begin{equation}\label{eq:main_problem} \tag{$\mathcal{P}_{\mu,\lambda}$}
\left\{
\begin{aligned}
    -\Delta u + (-\Delta)^s u - \mu \frac{u}{|x|^2} &= \lambda |u|^{p-2} u + |u|^{2^*-2} u \quad \text{in } \Omega \subset \mathbb{R}^N, \\
    u &= 0 \quad \text{in } \mathbb{R}^N \setminus \Omega,
\end{aligned}
\right.
\end{equation}
where $\om \subset \rnn$ is a bounded open set of class $C^{1,\al}$ for some $\al \in (0,1)$ such that $0 \in \om$, $\varepsilon \in (0,1]$, $0<s<1<p<2^*$ with $2^* = \frac{2N}{N-2}$, $\la >0$ is a parameter, $(-\Delta)^s$ is the fractional Laplacian operator defined as 
\begin{equation*}
    (-\Delta )^s u (x) = \text{P.V.} \int_{\rnn} \frac{u(x) - u(y)}{|x-y|^{N+2s}} dy, 
\end{equation*}
where P.V. is the Cauchy principle value and $\mu \in (0, \bar \mu)$ where $\bar \mu = \left( \frac{N-2}{2} \right)^2$ is the optimal constant in the Hardy's inequality \cite{adimurthi_hardy_inequality} which is given by
\begin{equation*}
     \int_{\om} \frac{|u|^2}{|x|^2}~ dx \leq \frac{1}{\bar \mu} \int_{\om} |\grad u |^2 ~ dx \qquad \text{for all } u \in C_c^{\infty}(\om).
\end{equation*}
The combination of local and nonlocal operators has recently emerged as a prominent area of research because of its wide-ranging and increasingly recognized applications. In fields such as finance and control theory, modelling often requires incorporating both diffusion and jump components in the underlying Markov processes. This dual nature introduces significant challenges, as the process operates on two distinct scales: the diffusion component dominates at small scales, while the jump component becomes more influential at larger scales. Such operators also play a crucial role in population dynamics. A rigorous mathematical treatment of these models can be found in~\cite{valdinoci_mixed_application}. 

The analysis of PDEs involving singular potentials is of intrinsic interest. In quantum mechanics, for instance, the Hardy potential characterizes motion and interactive properties such as repulsion and attraction between charged particles (see~\cite{frank_hardy_applications}). For a comprehensive treatment of the Hardy potential and its wide-ranging applications, we refer the interested reader to the monograph~\cite{peral_hardy_applications}. It is worth noting that in the context of mixed operators, the Hardy potential is typically chosen to coincide with that associated with the local operator. This choice is motivated by the fact that, for the mixed operator, all potentials of the form \( |x|^{-t} \) with \( t \in [2s, 2] \) are admissible; see Lemma~3.1 of~\cite{malhotra2025eigenvalues}. However, among these, the local Hardy potential exhibits the strongest singularity, making it the most natural candidate for perturbative analysis. A rigorous discussion on this can be found in~\cite{biagi2024mixed_hardy}.

Motivated by the foundational role of the Hardy potential, the study of Brezis-Nirenberg-type problems for Laplacian operator with Hardy potential was initiated by Jannelli~\cite{jannelli_hardy_starting}, who investigated the following problem:
\begin{equation}\label{eq:local_problem} 
\left\{
\begin{aligned}
    -\Delta u  - \mu \frac{u}{|x|^2} &= \lambda  u + |u|^{2^*-2} u \quad \text{in } \Omega \subset \mathbb{R}^N, \\
    u &= 0 \quad \text{on } \partial\Omega.
\end{aligned}
\right.
\end{equation}
The author proved that if \( 0 < \mu \leq \bar{\mu} - 1 \), then problem~\eqref{eq:local_problem} admits a positive solution for all \( \lambda \in (0, \lambda_1) \). In contrast, if \( \bar{\mu} - 1 < \mu < \bar{\mu} \) and \( \Omega = B_1(0) \) is the unit ball, there exists a threshold \( \lambda_* \) such that problem~\eqref{eq:local_problem} admits a positive solution if and only if \( \lambda \in (\lambda_*, \lambda_1) \), where \( \lambda_1 \) denotes the first eigenvalue of the Laplacian operator with the Hardy potential. This result highlights that every dimension can become critical when \( \mu \) is close to \( \bar{\mu} \), specifically when \( \mu \in (\bar{\mu} - 1, \bar{\mu}) \). This stands in stark contrast to the classical Brezis-Nirenberg problem with \( \mu = 0 \) (see \cite{brezis_1983}), where only \( N = 3 \) is critical. This phenomenon is explained by the guiding principle in~\cite{jannelli_hardy_starting}, which asserts that a spatial dimension is critical for a linear elliptic operator \( \mathcal{L} \) if and only if \( \mathcal{L} \) admits at least one Green function \( G(x_0, x) \in L^2_{\text{loc}}(\mathbb{R}^N) \). Subsequently, Cao and Peng~\cite{cao_peng2003signchanging} established the existence of sign-changing solutions to problem~\eqref{eq:local_problem} for all \( \lambda \in (0, \lambda_1) \). 
In a related direction, Ferrero and Gazzola~\cite{ferrero2001existence} replaced the linear perturbation term \( \lambda u \) with a more general subcritical nonlinearity \( g(x,u) \) and established the existence of solutions under suitable conditions. For further results in the local case, we refer the reader to~\cite{pigong_local_p_hardy} and references therein. \\
Transitioning to the purely nonlocal case, various studies have been conducted to understand perturbations of the Hardy potential in the context of the fractional Laplacian. Dipierro et al.~\cite{dipierro2016fractionalwithHardy} proved the existence of extremals for the fractional Hardy-Sobolev inequality and investigated several of their qualitative properties. In particular, they derived asymptotic estimates via a suitable change of variables and analysis of the resulting transformed equation. Ghoussoub et al.~\cite{ghoussoub_fractional_hardy} considered the following problem involving the fractional Hardy-Schrödinger operator:
\begin{equation}\label{eq:nonlocal_problem}
\left\{
\begin{aligned}
    (-\Delta)^s u  - \mu \frac{u}{|x|^{2s}} &= \lambda u + \frac{|u|^{2^*_s(\alpha)-2} u}{|x|^{\alpha}} \quad \text{in } \Omega \subset \mathbb{R}^N, \\
    u &= 0 \quad \text{on } \mathbb{R}^N \setminus \Omega,
\end{aligned}
\right.
\end{equation}
where \( \mu \in [0, \Lambda_{N,s}] \), \( \lambda \in (0, \lambda_s(\mu)) \) with \( \lambda_{s,\mu} \) denoting the first eigenvalue of the Hardy-Schrödinger operator, \( 0 \leq \alpha < 2s < N \), and \( 2^*_s(\alpha) = \frac{2(N - \alpha)}{N - 2s} \). By introducing the concept of the \emph{internal mass} of the domain, they addressed the critical case and established the existence of least energy solutions to problem~\eqref{eq:nonlocal_problem} under various conditions on the Hardy term. Later on, Shang et al.~\cite{shang_zhang_fractional_laplacian} incorporated weights into the nonlinear terms and obtained results concerning the existence and multiplicity of solutions.\\
Unlike purely local or purely nonlocal cases, the study of mixed local–nonlocal operators involving the Hardy potential is still in the developmental stage. Biagi et al.~\cite{biagi2024mixed_hardy} investigated the existence, uniqueness, and optimal summability of solutions. Malhotra~\cite{malhotra2025eigenvalues} studied the Fučík spectrum and shape optimization problems for the first two eigenvalues of such mixed operators. 

On the other hand, in the absence of the Hardy potential, Brezis–Nirenberg-type results for the mixed linear operator were first established in~\cite{biagi2022brezis} and subsequently generalized to mixed quasilinear operator by Silva et al. in~\cite{silva2024mixed}. 

Inspired by the aforementioned work, we started our study with the analysis of the following problem
\begin{equation}\label{eq:main_eq}
\left\{
\begin{aligned}
    -\Delta u + \varepsilon(-\Delta)^s u - \mu \frac{u}{|x|^2} &=|u|^{2^*-2} u \quad \text{in } \Omega \subset \mathbb{R}^N, \\
    u &= 0 \quad \text{in } \mathbb{R}^N \setminus \Omega.
\end{aligned}
\right.
\end{equation}
The solutions of the equation \eqref{eq:main_eq} are intimately connected with the minimizers of the following ratio
\begin{equation} \label{eq:spectral_ratio_frac}
S_{\mu, s,\varepsilon}(\Omega) := \inf_{u \neq 0} 
\frac{\displaystyle\int_{\Omega} |\nabla u|^2 - \mu \int_{\Omega} \frac{|u|^2}{|x|^2} + \varepsilon[u]_{s}^2}
{\left(\displaystyle\int_{\Omega} |u|^{2^*} \right)^{\frac{2}{2^*}}}.
\end{equation}
Thus, the study of the problem \eqref{eq:main_eq} is completed once we find the minimizers. Unlike the purely local \cite{terracini_minimizers} or nonlocal case \cite{cotsiolis2004best_fractional}, the minimizers for combination of local and nonlocal terms does not exist due to the lack of common scaling invariance. 

This can be seen in the following theorem
\begin{theorem} Let $s \in (0,1)$, $\varepsilon \in (0,1]$ be fixed and let $\Omega \subseteq \mathbb{R}^N$ be an open set. Then
\begin{equation}\label{eq:best_constant_equality}
S_{\mu, s, \varepsilon}(\Omega) = S_{\mu}(\Omega) = S_{\mu},
\end{equation}
where
\begin{equation*} 
S_{\mu}(\Omega) := \inf_{u \neq 0} 
\frac{\displaystyle\int_{\Omega} |\nabla u|^2 - \mu \int_{\Omega} \frac{|u|^2}{|x|^2}}
{\left(\displaystyle\int_{\Omega} |u|^{2^*} \right)^{\frac{2}{2^*}}}.
\end{equation*}
Moreover, the optimal constant $S_{\mu, s, \varepsilon}(\Omega)$ in equation \eqref{eq:spectral_ratio_frac} is never attained and independent of $\om$.
\end{theorem}
The proof of the above theorem follows in the same spirit of Theorem $1.1$ and $1.2$ of \cite[Theorem $1.1$]{biagi2024mixed_hardy}. \\
In fact, a complete classification of positive solutions of the minimization problem for $S_{\mu}(\om)$ is done in \cite{olivia_classfication_localp} via the moving plane method. They proved that all the solutions are radial and radially decreasing about the origin.\\
Since the best constant in the minimization problem \eqref{eq:spectral_ratio_frac} is not achieved. The corresponding problem \eqref{eq:main_eq} does not possesses the groundstate solutions. This motivates the study of the following perturbed problem \eqref{eq:main_problem_ep}
\begin{equation}\label{eq:main_problem_ep} \tag{$\mathcal{P}_{\mu,\lambda}^{\varepsilon}$}
\left\{
\begin{aligned}
    -\Delta u + \varepsilon(-\Delta)^s u - \mu \frac{u}{|x|^2} &= \lambda |u|^{p-2} u + |u|^{2^*-2} u \quad \text{in } \Omega \subset \mathbb{R}^N, \\
    u &= 0 \quad \text{in } \mathbb{R}^N \setminus \Omega,
\end{aligned}
\right.
\end{equation}
where $\varepsilon \in (0,1]$. \\
We emphasize that introducing the parameter \( \varepsilon \) in front of the fractional Laplacian term in problem~\eqref{eq:main_problem_ep} is essential (at least in the sublinear case) in order to establish the existence of solutions via the standard variational approach, as adapted from~\cite{brezis_1983}. The key idea is to convert the existence question to the analysis of the following minimization problem
\begin{equation*} 
S_{\mu, s,\lambda}(\Omega) := \inf_{u \neq 0} 
\frac{\displaystyle\int_{\Omega} |\nabla u|^2 \, dx - \mu \int_{\Omega} \frac{|u|^2}{|x|^2} \, dx +  [u]_{s}^2 - \lambda \int_{\Omega} |u|^2 \, dx}
{\left(\displaystyle\int_{\Omega} |u|^{2^*} \, dx \right)^{\frac{2}{2^*}}}.
\end{equation*}
A solution to problem~\eqref{eq:main_problem} exists provided that $S_{\mu, s,\lambda}(\Omega) < S_{\mu, s}$, where $ S_{\mu, s}:=S_{\mu, s,1}$. The standard method to verify this inequality involves selecting a suitable test function and evaluating the quotient to show it lies strictly below the threshold. Since \( S_{\mu, s} = S_\mu \), a natural choice is to choose a minimizer for $S_{\mu}$ which is attained by a family of functions \( U_{\varepsilon} \)(defined explicity later). However, one encounters the following asymptotic estimates
\begin{equation*}
    [U_{\varepsilon}]_s^2 = O\left(\varepsilon^{2(1-s)\frac{\sqrt{\bar\mu}}{\sqrt{\bar\mu - \mu}}}\right), \quad 
    \int_{\Omega} |U_{\varepsilon}|^2 \, dx = O\left(\varepsilon^{2\frac{\sqrt{\bar\mu}}{\sqrt{\bar\mu - \mu}}}\right) 
    \quad \text{as } \varepsilon \to 0^+.
\end{equation*}
These estimates reveal that the contribution of the Gagliardo seminorm $[U_{\varepsilon}]_s$ (defined in Section \ref{sec:notations}) becomes non-negligible as \( \varepsilon \to 0^+ \), thereby obstructing the inequality \( S_{\mu, s,\lambda} < S_{\mu, s} \) and making the existence of a minimizer via this approach appear infeasible.

The aim of this paper is twofold. Firstly, the equality in \eqref{eq:best_constant_equality} naturally raises the question of the nature of minimizers of \(S_{\mu, s, \varepsilon}(\Omega)\) with a perturbation, and how they differ from their local counterparts. This motivates the study of the exact asymptotic behavior of solutions to \eqref{eq:main_problem_ep} near the origin.\\
To establish the asymptotic estimates, the key idea is to transform the original problem \eqref{eq:main_problem_ep} into the reformulated problem \eqref{eq:transformed_equation}(defined later), which involves working within radial Sobolev spaces. For this transformed problem, a Harnack inequality and uniform estimates are derived, which play a crucial role in obtaining the lower and upper asymptotics, respectively.

\begin{theorem}\label{thm:asympototic_estimates}
Let $0 < \mu < \bar\mu$ and $1<p<2^*$. Then for any weak solution $u$ of \eqref{eq:main_problem_ep} there admits two positive constants $M_1$ and $M_2$ independent of $\varepsilon$ such that
\begin{equation*}
M_1 |x|^{- \left( \sqrt{\bar{\mu}} - \sqrt{\bar{\mu} - \mu} \right)} \leq u(x) \leq M_2 |x|^{- \left( \sqrt{\bar{\mu}} - \sqrt{\bar{\mu} - \mu} \right)} \quad \text{for all } x \in B_{r_0}(0) \subset \Omega,
\end{equation*}
with some $r_0 > 0$ sufficiently small.
\end{theorem}



Secondly, we turn our analysis to the effect of perturbations on the existence of solutions. We begin with the case of a linear perturbation, which requires knowledge of the first eigenvalues of the fractional operator and the mixed operator with Hardy potential, defined respectively as
\begin{equation} \label{eq:first_eigenvalue_frac}
   \lambda_{1,s} := 
  \inf\big\{[u]^2_s:\,\text{$u\in C_c^\infty(\Omega)$ and $\|u\|_{L^2(\rnn)} = 1$}\big\};
\end{equation}
\begin{equation} \label{eq:first_eigenvalue}
\lambda_{1} := 
  \inf\left\{ \int_{\om} |\grad u|^2 dx + [u]_s^2 - \mu \int_{\om}\frac{|u|^2}{|x|^2}dx:\,\text{$u\in C_c^\infty(\Omega)$ and $\|u\|_{L^2(\rnn)} = 1$}\right\}.
\end{equation}
For a comprehensive results related with these eigenvalues we refer to \cite[Remark $4.4$]{biagi2022brezis}. Then we have the following theorem concerning the solutions of the linear problem \eqref{eq:main_problem}
\begin{theorem}\label{thm:existence_linear_case}
Let $\Omega \subseteq \mathbb{R}^N$, $\varepsilon = 1$ and $p = 2$. Then the following holds
\begin{enumerate}
    \item For every $0 < \lambda \leq \lambda_{1, s}$, there does not exists a solution $u \in \mathcal{B} \subset L^{2^*}(\mathbb{R}^N)$ of \eqref{eq:main_problem}, where
    \begin{equation}\label{eq:linear_ball_NE}
    \mathcal{B} := \left\{ u \in L^{2^*}(\mathbb{R}^N) : \| u \|_{L^{2^*}} \leq  S_{\mu}^{\frac{N-2}{4}} \right\}.
    \end{equation}
    \item There exists a parameter $\lambda^* \in [\lambda_{1, s}, \lambda_1)$ such that the problem \eqref{eq:main_problem} possesses at least one solution if $\lambda \in (\lambda^*, \lambda_1)$. 
    \item There does not exist a solution to the problem \eqref{eq:main_problem} if $\lambda > \lambda_1$.
\end{enumerate}
\end{theorem}

For the superlinear case, i.e., $p \in (2,2^*)$, we are interested in the existence of a positive solution of the problem \eqref{eq:main_problem}. We apply the mountain pass theorem to prove the existence of a solution. However, the presence of nonlocal term added another realm of difficulties. A perturbation close to linear power is not sufficient to guarantee the existence of the solution. In fact we require a perturbation with higher exponent to tackle the presence of fractional term. For this purpose we define the following crucial parameter
\begin{equation}\label{eq:superlinear_parameter}
    \beta_{\mu,N,s} = \min \left\{ N-2, \frac{2(1-s)\sqrt{\bar\mu}}{\sqrt{\bar\mu - \mu}} \right\}.
\end{equation}
Thus, we have the following existence result.
\begin{theorem}\label{thm:superlinear_pb}
Let $p \in (2,2^*)$. Then there exists a $\la_0 >0$ such that the problem \eqref{eq:main_problem} has a positive solution in the following cases
\begin{enumerate}
    \item $ 0 < \la < \la_0$ and $\beta_{\mu,N,s} > \frac{\sqrt{\bar\mu}(N - p \sqrt{\bar\mu})}{\sqrt{\bar\mu - \mu}}$,
    \item $\la \geq \la_0$,
\end{enumerate}
where $\beta_{\mu,N,s}$ is defined as in \eqref{eq:superlinear_parameter}.
\end{theorem}
As we note that the first case in the above only occurs when the exponent $p$ is sufficiently large enough. For perturbations with lower order exponent, the existence of solutions are proved only for large values of $\la$.\\

We end our study by adding the perturbation of sublinear nature, i.e., $p \in (1,2)$. This nonlinearity greatly influences the topology of the associated functional of the problem \eqref{eq:main_problem_ep} and leads to the existence of two nonnegative multiple solutions in the spirit of \cite{abc_laplacian}. However, the solutions are not expected to be bounded in a small neighbourhood of zero. This requires us to work with minimal solutions. Once the minimal solutions are obtained, we are able to prove the following multiplicity result with the help of the Mountain pass theorem. 
\begin{theorem}\label{thm:sublinear_two_solutions}
Suppose \(0 < \mu < \bar{\mu} - 1\). Then there exists \(\Lambda > 0\) such that  
\begin{enumerate}
    \item for all \(\lambda \in (0, \Lambda)\), the problem \eqref{eq:main_problem_ep} admits a minimal solution \(u_{\lambda}\), and these minimal solutions are increasing with respect to \(\lambda\);  
    \item for \(\lambda = \Lambda\), the problem \eqref{eq:main_problem_ep} admits at least one weak solution;  
    \item for \(\lambda > \Lambda\), the problem \eqref{eq:main_problem_ep} does not admit any solution.  
\end{enumerate}
Moreover, if \(\lambda \in (0, \Lambda)\), the problem admits a second positive solution distinct from the first one.
\end{theorem}

One of the main novelties of this paper is the derivation of blow-up estimates for the Gagliardo norm of minimizers of \(S_{\mu}\). These estimates are obtained by employing certain inequalities from \cite{giovanni_fractional_book}. Another difficulty arises in establishing the Harnack inequality, where obtaining estimates for the Hardy term proves challenging; this issue is addressed by applying a crucial lemma from \cite{stein_weiss1957fractional}. Finally, the presence of the Hardy term results in the lack of equivalence of minimizers in the \(C^1\) topology, which necessitates a delicate analysis in order to prove the existence of local minima of solutions in the sublinear case.

The paper is organized as follows. In Section~\ref{sec:notations}, we begin by introducing the basic notations and preliminaries required throughout the paper. In addition, several auxiliary lemmas are stated and proved to support the main results. Then we derive asymptotic estimates for solutions of problem~\eqref{eq:main_problem_ep} by employing a suitable transformation in Section~\ref{sec:asymptotics} and end it by proving a strong maximum principle. 
In Section~\ref{sec:linear_sublinear}, we start with examining the case of linear perturbations and establish both existence and nonexistence results, depending on the value of the parameter relative to the first eigenvalue of the underlying operator. Then we consider the case of superlinear perturbations. Applying the Mountain Pass Theorem, we show the existence of solutions under appropriate conditions on the exponent of the nonlinearity. Finally the section~\ref{sec:sublinear} is devoted to the analysis of sublinear perturbations, where the existence of two positive solutions is established.

\section{Notations and Preliminaries}\label{sec:notations}

In this section we setup the function spaces which are used to study our problem \eqref{eq:main_problem_ep} and state some lemmas and prove some propositions required in our analysis. First of all, let us fix some notations
\begin{enumerate}
    \item[$(i)$] The constants will be denoted by $C$ and they are allowed to vary within a single line or formula. 
    \item[$(ii)$] If $A$ is any measurable set in $\rnn$, then $|A|$ denotes the $N$-dimensional Lebesgue measure.
    \item[$(iii)$] $B_R(x)$ denotes the $N$-dimensional ball in $\rnn$ centered at $x$ having radius $R$. If it is centered at $0$, then we denote it by $B_R$.
    \item[$(iv)$] For $1 \leq p < \infty$, we denote and define the $L^p(\om)$ norm as $\lv u \rv_{L^p(\om)}^p := \int_{\om} |u|^p dx$.
    \item[$(v)$] $C_c^\infty(\om):= \{ u \in C^{\infty}(\om) : Supp(u) \text{ is compact in } \om \}$ and $C_{c,rad}^\infty(\om) := \{ u \in C_c^\infty(\om) : u \text{ is radial}\}$.
\end{enumerate}

Now, let $\om \subset \rnn$ be a bounded domain and $\varepsilon \in (0,1]$. Then we define the following function space
\begin{equation*}
\X(\Omega) := \left\{ u \in H^1(\mathbb{R}^N) \,:\, u_{|\Omega} \in H_0^1(\Omega) \text{ and } u \equiv 0 \text{ a.e. in } \mathbb{R}^N \setminus \Omega \right\}.
\end{equation*}
It is a Hilbert space endowed with the norm 
\begin{equation*}
    \lv u \rv_{\X,\varepsilon} := \left( \int_{\om} |\grad u|^2 dx +  \varepsilon [ u ]_s^2\right)^{\frac{1}{2}},
\end{equation*}
where $[u]_s^2 = \iint_{\rtwon} \frac{|u(x) - u(y)|^2}{|x-y|^{N +2s}}dy dx$ is known as Gagliardo seminorm induced from the inner product $\langle u, v\rangle_s = \iint_{\rtwon} \frac{(u(x) - u(y))(v(x) - v(y))}{|x - y|^{N  + 2s}}dy dx$. We note that $\lv \cdot \rv_{\X,\varepsilon}$ is equivalent to the usual norm $\lv \cdot \rv_{\X} \left( := \lv \cdot \rv_{\X,1} \right)$. Moreover, $\X(\om)$ is reflexive and separable with respect to the norm $\lv \cdot \rv_{\X,\varepsilon}$. \\
A function $u \in \X(\om)$ is said to be a weak solution of the problem \eqref{eq:main_problem_ep} if for every $v \in \X(\om)$, it holds 
\begin{equation*}
    \int_{\om} \grad u \grad v \, dx + \varepsilon \langle u , v\rangle_s - \mu \int_{\om} \frac{u v}{|x|^2}dx = \int_{\om} |u|^{2^* -2}u v dx + \int_{\om} |u|^{p-2}u v dx.
\end{equation*}
In addition, if the equality is replaced with $\geq(\leq)$ then it is known as weak supersolution(subsolution). 
The weak solution to problem \eqref{eq:main_problem_ep} corresponds to a critical point of the energy functional $\Je: \X(\Omega) \to \mathbb{R}$ given by
\begin{equation*}
\Je(u) = \frac{1}{2} \int_{\Omega} |\nabla u|^2 - \frac{\mu}{2} \int_{\Omega} \frac{|u|^2}{|x|^2} dx
+ \frac{\varepsilon}{2} [u]_{s}^2 
- \frac{1}{2^*} \int_{\Omega} |u|^{2^*} dx
- \frac{\lambda}{p} \int_{\Omega} |u|^{p}dx.
\end{equation*}
It is clear that $\Je \in C^1(\X,\real)$. 

\begin{definition}
    A sequence $\{ u_n \} \subset \X(\Omega)$ is said to be a Palais-Smale $(PS)$ sequence for a functional $\mathcal{J}$ at level $\beta$, if $\mathcal{J}(u_n) \to \beta$ in $\R$ and $\mathcal{J}^{\prime}(u_n) \to 0$ in $(\X(\Omega))'$ as $n \to \infty$. The function $\mathcal{J}$ is said to satisfy $(PS)$ condition at level $\beta$, if every $(PS)$ sequence for $\mathcal{J}$ at level $\beta$ admits a convergent subsequence.
\end{definition}

Also, we would like to mention a very crucial lemma from \cite{stein_weiss1957fractional}.
\begin{lemma}[Lemma $3.5$ of \cite{stein_weiss1957fractional}]\label{lem:stein_strong_lemma}
Let 
\begin{equation*}
    V_{\de}(f) = |x|^{-N + \delta} \int_{\{|y| < |x|\}} |y|^{- \delta} f(y) \, dy
\end{equation*}
where $1< q < \infty$ and $\de < \frac{N(q-1)}{q}$. Then $|V_{\de}(f)| \leq C |x|^{-\frac{N}{q}} \lv f \rv_{L^q}$ for some $C>0$.
\end{lemma}

Finally, we end this section by mentioning some useful inequalities borrowed from \cite{brasco_second_eigenvalue}
\begin{lemma}\label{ineq:A1}
    Let $f$ be a convex function, then 
    \begin{equation*}
        (a-b)[A f^{\prime}(a) - B f^{\prime}(b)] \geq (f(a) - f(b))(A - B)
    \end{equation*}
for all $a, b \in \real$ and $A, B \geq 0$.
\end{lemma}

\begin{lemma}\label{ineq:A2}
Let $g : \real \to \real $ be an increasing function and $G(t) = \int_0^t g^{\prime} (\tau)^{\frac{1}{2}} d\tau$. Then
\begin{equation*}
    (a - b)(g(a) - g(b)) \geq |G(a) - G(b)|^2.
\end{equation*}
\end{lemma}

\begin{lemma}\label{ineq:A3}
Let $g : \real \to \real $ be an decreasing function and $H(t) = -\int_0^t g^{\prime} (\tau)^{\frac{1}{2}} d\tau$. Then
\begin{equation*}
    (a - b)(g(b) - g(a)) \geq |H(a) - H(b)|^2.
\end{equation*}
\end{lemma}

\section{Qualitative Properties}\label{sec:asymptotics}

In this section we study the behaviour of solutions of \eqref{eq:main_problem_ep} near the origin. We start with transforming the problem \eqref{eq:main_problem_ep} into a suitable problem. After that Harnack inequality and uniform estimates are obtained. Finally, we end the section by proving a strong maximum principle.\\
We start this section by introducing the Sobolev inequality and Poincar\'{e} inequality in the radial case.

\begin{lemma}\label{lem:weighted_sobolev_inequality}
Let $d \in  (0, \frac{N-2}{2})$. Then there exists a positive constant $C>0$ such that for all $u \in C_c^{\infty}(B_R)$ the following inequality holds
\begin{equation*}
        \left( \frac{1}{|B_R|_{d\mu}}\int_{B_R} \frac{|u|^q}{|x|^{2d}} dx\right)^{\frac{1}{q}} \leq C \left( \frac{R^2}{|B_R|_{d \mu}} \int_{B_R} \frac{|\grad u |^2}{|x|^{2d}}dx\right)^{\frac{1}{2}},
\end{equation*}
where $q = \frac{2 (N- 2d)}{N - 2d - 2}$ and $|B_R|_{d \mu} = \int_{B_R} \frac{dx}{|x|^{2d}}$.
\end{lemma}
We note that the above inequality also holds for lower exponents \( 1 \leq q_1 < q \), which is a direct consequence of Jensen's inequality. \\
To prove the above inequality we require the following lemma.
\begin{lemma}\label{lem:radial_estimate}
For all $u \in C_{c, \mathrm{rad}}^{\infty}(B_1)$, the following holds:
\begin{enumerate}
    \item[$(i)$]  $
|u(x)| \;\leq\; C \left( \int_{B_1} \frac{|\nabla u|^2}{|y|^{2d}} \, dy \right)^{\tfrac{1}{2}}
\, r^{-\tfrac{(N-2d-2)}{2}},$
\item[$(ii)$] $
\left( \int_{B_1} \frac{|u|^q}{|x|^{2d}} \, dx \right)^{1/q} 
\;\leq\; C \left( \int_{B_1} \frac{|\nabla u|^2}{|x|^{2d}} \, dx \right)^{1/2},$
\end{enumerate}
where $d = \sqrt{\bar{\mu}} - \sqrt{\bar{\mu} - \mu}, \, r = |x| < 1$ and $q = \frac{2(N-2d)}{N-2d-2}$.
\end{lemma}

\begin{proof}
Since $u \in C_{c,\mathrm{rad}}^{\infty}(B_1)$, we have
\begin{equation*}
-u(x) = - u(r) = u(1) - u(r) = \int_{r}^{1} u'(s) \, ds .
\end{equation*}
This together with the H\"{o}lder inequality yields
\begin{equation*}
\begin{aligned}
|u(x)| 
&\leq \int_{r}^{1} |u'(s)| \, ds \\
&\leq \left( \int_{r}^{1} |u'(s)|^2 \, s^{N-2d-1} \, ds \right)^{\tfrac{1}{2}}
\left( \int_{r}^{1} s^{-(N-2d-1)} \, ds \right)^{\tfrac{1}{2}} \\
& \leq \omega_N^{-1/2} 
\left( \int_{B_1} \frac{|\nabla u|^2}{|y|^{2d}} \, dy \right)^{\tfrac{1}{2}}
\left( \frac{r^{-(N-2-2d)}}{N-2-2d} \right)^{\tfrac{1}{2}}\\
& =\; C 
\left( \int_{B_1} \frac{|\nabla u|^2}{|y|^{2d}} \, dy \right)^{\tfrac{1}{2}} 
r^{-\tfrac{(N-2-2d)}{2}} .
\end{aligned}
\end{equation*}
For the second part, we apply Cauchy-Schwarz inequality and part $(i)$ of Lemma~\ref{lem:radial_estimate} to obtain
\begin{equation*}
\begin{aligned}
\int_{B_1} \frac{|u|^q}{|x|^{2d}} \, dx 
&= \omega_N \int_{0}^{1} |u(r)|^q \, r^{N-1-2d} \, dr \\
&= -\omega_N \int_{0}^{1} q |u(r)|^{q-2} u(r) u'(r) \, \frac{r^{N-2d}}{N-2d} \, dr \\
&\leq \frac{q \, \omega_N}{N-2d} 
\int_{0}^{1} \big( |u'(r)| r^{\tfrac{N-2d-1}{2}} \big)
\big( |u(r)|^{q-1} r^{\tfrac{N-2d+1}{2}} \big) \, dr \\
& \leq \frac{q \, \omega_N}{N-2d} 
\left( \int_{0}^{1} |u'(r)|^2 r^{N-2d-1} \, dr \right)^{1/2}
\left( \int_{0}^{1} |u(r)|^{2(q-1)} r^{(N-2d+1)} \, dr \right)^{1/2}\\
& \leq \frac{C \, q }{N-2d} 
\left( \int_{B_1} \frac{|\nabla u|^2}{|x|^{2d}} \, dx \right)^{q/4}
\left( \int_{B_1} \frac{|u|^q}{|x|^{2d}} \, dx \right)^{1/2}.
\end{aligned}
\end{equation*}
Rearranging terms, we finally deduce
\begin{equation*}
\left( \int_{B_1} \frac{|u|^q}{|x|^{2d}} \, dx \right)^{1/q}
\;\leq\; \left( \frac{C \, q}{N-2d} \right)^{2/q}
\left( \int_{B_1} \frac{|\nabla u|^2}{|x|^{2d}} \, dx \right)^{1/2}. \qedhere
\end{equation*}

\end{proof}

\begin{proof}[Proof of Lemma \ref{lem:weighted_sobolev_inequality}]
Let $u \in C_{c, \mathrm{rad}}^{\infty}(B_R)$ be radial. Take $v(x) = u(Rx)$ and apply part $(ii)$ of Lemma \ref{lem:radial_estimate} to get 
\begin{equation*}
    \left( \int_{B_R} \frac{|u|^q}{|x|^{2d}} \, dx \right)^{1/q} 
\;\leq\; C \left( \int_{B_R} \frac{|\nabla u|^2}{|x|^{2d}} \, dx \right)^{1/2}.
\end{equation*}
Now, let $  u \in C_c^{\infty}(B_R)$ and $u^*$ denotes its symmetric rearrangement, then using the Rearrangement inequality (see Theorem $3.4$ and property $(v)$ in section $3.3$ of \cite{liebloss2001analysis}) and Theorem $8.1$ of \cite{alvino2017radial_polya} with parameters $l = 2d$ and $ k = 0$ (this choice satisfies condition $(ii)$ of Theorem $1.1$), we deduce
\begin{equation*}
\begin{aligned}
 \left( \int_{B_R} \frac{|u|^q}{|x|^{2d}} \, dx \right)^{1/q}  & \leq    \left( \int_{B_R} (|u|^q)^* \left(\frac{1}{|x|^{2d}} \right)^* \, dx \right)^{1/q}  = \left( \int_{B_R} |u^*|^q \frac{1}{|x|^{2d}}  \, dx \right)^{1/q} \\  &\leq C \left( \int_{B_R} \frac{|\nabla u^*|^2}{|x|^{2d}} \, dx \right)^{1/2} \leq C \left( \int_{B_R} \frac{|\nabla u|^2}{|x|^{2d}} \, dx \right)^{1/2}.
\end{aligned}
\end{equation*}
From this the result follows.
\end{proof}

Further, since the weights $|x|^{-\delta}$ are admissible weights in the sense of \cite{heinonen_potential_theory}, we state the following Poincar\'{e} inequality
\begin{lemma}\label{lem:weighted_poincare_inequality}
 Let $\delta > 0$. Then there exists a constant $C > 0$ such that for all $u \in C^{\infty}(B_R)$ the following inequality holds
 \begin{equation*}
     \int_{B_R} | u - (u)_{B_R}|^2 \frac{dx}{|x|^\delta} \leq C R^2 \int_{B_R} |\grad u|^2 \frac{dx}{|x|^{\delta}},
 \end{equation*}
 where $(u)_{B_R} = \frac{1}{|B_{R}|_{d\delta}} \int_{B_{R}} u \frac{dx}{|x|^{\delta}}$ with $|B_R|_{d \delta} = \int_{B_R} \frac{dx}{|x|^{\delta}}$.
\end{lemma}
For general domains, we have the following Sobolev inequality with weights
\begin{lemma}
Let $d \in \left(0, \frac{N-2}{2} \right)$. Then there exists a positive constant $C>0$ such that for all $u \in C_c^{\infty}(B_R)$ the following inequality holds  
\begin{equation}\label{ineq:sobolev_inequality_weight}
 \left( \int_{\om} |u |^{2^*} \frac{dx}{|x|^{2^*d}}\right)^{2/2^*}   \leq  C \int_{\om}  |\grad u |^2 \frac{dx}{|x|^{2d}}.
\end{equation}
\end{lemma}
\begin{proof}
The proof follows from picking the following parameters in \cite{caffarelli1984interpolation_sobolev}
\begin{equation*}
a = \frac{1}{2}, \; \gamma = \beta = \alpha = \sigma = - d, \; r = q = 2^*, \text{ and } p=2. \qedhere
\end{equation*}
\end{proof}

Now we state the ground state representation borrowed from \cite[Formula $(4.3)$]{frank2008hardy} and refined in \cite[Lemma $3.2$]{ghoussoub_fractional_hardy}.\\
\textbf{Ground state representation:} For $0 < s < 1$, $N > 2$, and $0 < d < \frac{N - 2}{2}$, for $u \in C_c^\infty(\mathbb{R}^N \setminus \{0\})$, we define
\begin{equation*}
    w(x) = |x|^{d} u(x).
\end{equation*}
\begin{equation*}
\begin{aligned}
    \frac{C_{N,2s}}{2} \iint_{\rtwon} \frac{|u(x) - u(y)|^2}{|x - y|^{N + 2s}} \,dx\,dy 
    &= \psi_{N,2s}(d) \int_{\mathbb{R}^N} \frac{u^2(x)}{|x|^{2s}} \,dx + \frac{C_{N,2s}}{2} \iint_{\rtwon} \frac{|w(x) - w(y)|^2}{|x - y|^{N + 2s}} \frac{dx}{|x|^{d}} \frac{dy}{|y|^{d}},
\end{aligned}
\end{equation*}
where
\begin{equation*}
    \psi_{N,2s}(d) = 2^{2s} \frac{
        \Gamma\left( \frac{N  {-d}}{2} \right) 
        \Gamma\left( \frac{2s +{d}}{2} \right)
    }{
        \Gamma\left( \frac{N  {-d} - 2s}{2} \right)
        \Gamma\left( \frac{d}{2} \right)
    }, \text{ and } \quad 
     C_{N,2s} = \frac{
        2^{2s} \Gamma\left( \frac{N + 2s}{2} \right)
    }{
        \pi^{N/2} \left| \Gamma\left( -\frac{2s}{2} \right) \right|
    }.
\end{equation*}
Now, from \cite{chen2004multiple}, we have
\begin{equation*}
    \int_{\Omega}  |\nabla w(x)|^2 \,\frac{dx}{|x|^{2 d}} = \int_{\Omega} |\nabla u|^2 \, dx - \mu \int_{\Omega} \frac{u^2}{|x|^2} \,dx.
\end{equation*}
Substituting these into the identity obtained in \eqref{eq:main_problem_ep} after testing with $u$, we obtain the following ground state representation
\begin{equation*}
\begin{aligned}
    \int_{\Omega}  \frac{|\nabla w|^2}{|x|^{2d}} \,dx 
    &+ \varepsilon\frac{2 \psi_{N,2s}(d)}{C_{N,2s}} \int_{\Omega} \frac{ w^2(x)}{|x|^{2s}} \,\frac{dx}{|x|^{2 d}} \\
    &+ \varepsilon\iint_{\mathbb{R}^{2N}} \frac{|w(x) - w(y)|^2}{|x - y|^{N + 2s}} \frac{dx}{|x|^{d}} \frac{dy}{|y|^{d}}
    = \int_{\Omega} \frac{ w^{2^*}}{|x|^{2^* d}} \,dx + \lambda \int_{\Omega}  \frac{w^{q+1}}{|x|^{(q+1) d}} \,dx.
\end{aligned}
\end{equation*}
This holds for all $w(x) = |x|^{d} u(x)$ with $u \in C_c^\infty(\mathbb{R}^N)$. In fact, by Lemma $4.4$ of \cite{dipierro2016fractionalwithHardy}, it also holds for all $u \in H^1(\mathbb{R}^N) \cap H^s(\mathbb{R}^N)$.\\

In view of this ground state representation, if $u$ is a solution of \eqref{eq:main_problem_ep} then we make a change of variable defined by
\begin{equation*}
    u(x) = |x|^{-d} w(x) \quad \text{ with } d = \sqrt{\bar\mu} - \sqrt{\bar\mu - \mu},
\end{equation*}
where $\bar\mu= \left( \frac{N-2}{2} \right)^2$ being the critical Hardy constant. Then we deduce that $w$ satisfies the following equation as a weak solution
\begin{equation}\label{eq:transformed_equation}\tag{$\mathcal{T_{\mu,\la}^{\varepsilon}}$}
\left\{
\begin{aligned}
    -\mathrm{div}\left( \frac{\nabla w(x)}{|x|^{2 d}}\right) + \varepsilon C_c \frac{w}{|x|^{2s + 2d}} + \varepsilon (-\Delta_{d})^s w &= \frac{w^{2^* - 1}}{|x|^{2^* d}} + \lambda  \frac{w^q}{|x|^{(q+1) d}} \quad \text{in } \Omega,\\
    w &= 0 \quad \text{ in } \rnn \setminus \om,
\end{aligned}
\right.
\end{equation}
where the constant \( C_0 = \frac{2 \psi_{N,2s}(d)}{C_{N,2s}} \). \\
The operator \( (-\Delta_{d})^s \) is defined via the duality pairing
\begin{equation*}
    \langle (-\Delta_{d})^s v, \varphi \rangle_{s,d} = 
    \iint_{\rtwon} \frac{(v(x) - v(y))(\varphi(x) - \varphi(y))}{|x - y|^{N + 2s}}  \frac{dx}{|x|^{d}}  \frac{dy}{|y|^{d}}
\end{equation*}
for any \( \varphi \in \dot{H}^{s,d}(\mathbb{R}^N) \), where \( \dot{H}^{s,d}(\mathbb{R}^N) \) is the closure of \( C_c^\infty(\mathbb{R}^N) \) with respect to the norm
\begin{equation*}
    \|\varphi\|_{\dot{H}^{s,d}(\mathbb{R}^N)} = 
    \left( 
        \int_{\mathbb{R}^N} \frac{|\varphi(x)|^{2^*_s}}{|x|^{d 2^*_s}} \,dx 
    \right)^{\frac{1}{2^*}} 
    + 
    \left( 
        \iint_{\rtwon} \frac{|\varphi(x) - \varphi(y)|^2}{|x - y|^{N + 2s}}  \frac{dx}{|x|^{d}}  \frac{dy}{|y|^{d}} 
    \right)^{1/2},
\end{equation*}
where $2^*_s = \frac{2N}{N - 2s}$. Moreover, we need the space $ \dot{H}(\rnn,|x|^{-2d})$ to work with the transformed equation. It is defined as the closure of \( C_c^\infty(\mathbb{R}^N) \) with respect to the norm 
\begin{equation*}
    \|\varphi\|_{\dot{H}(\rnn,|x|^{-2d})} = 
    \left( 
        \int_{\mathbb{R}^N} \frac{|\varphi(x)|^{2^*}}{|x|^{d 2^*}} \,dx 
    \right)^{\frac{1}{2^*}} 
    + 
    \left( \int_{\rnn} | \grad \varphi |^2 dx \right)^2.
\end{equation*}
Having transformed the original equation, we now analyze the qualitative properties of its solutions. One of the key advantages of the transformed formulation is that boundedness of solutions becomes accessible. In order to establish this, we derive a weak Harnack inequality for solutions of the transformed problem~\eqref{eq:transformed_equation}. To simplify notations, we define
\begin{equation*}
    d\mu := \frac{dx}{|x|^{2d}}, \quad 
    d\nu := \frac{dx\,dy}{|x - y|^{N + 2s} |x|^{d} |y|^{d}},
\end{equation*}
and $|E|_{d\mu}$ is the measure of set $E$ with respect to measure $\mu$ given as $|E|_{d\mu} = \int_{E} \frac{1}{|x|^{2 d}} dx$.
\begin{theorem}\label{thm:harnack_inequality_transformed}
Let $d \in (0,\sqrt{\bar{\mu}})$ and $w \in \dot{H}^{s,d}(\mathbb{R}^N) \cap \dot{H}(\rnn,|x|^{-2d})$ be a weak solution of \eqref{eq:transformed_equation}. Then, there exist $r_0 > 0$ such that for $\eta \in [1,\frac{N-2}{2})$ the following inequality holds
\begin{equation*}
        \left( \frac{1}{|B_r|_{d\mu}}\int_{B_r} w^{\eta} \, d\mu \right)^{1/\eta} \leq C \inf_{B_{3r/2}} w,
\end{equation*}
for all $r < r_0$.
\end{theorem}
To prove this Harnack inequality, we first prove a series of lemmas. The first result in the direction is the following lemma, known as the \textit{propagation of positivity}

\begin{lemma}[Propagation of Positivity]\label{lem:propagation_of_positivity}
Let \( w > 0 \) in \( B_R(0) \), with \( 0 < R \leq 1 \), be a supersolution to equation~\eqref{eq:transformed_equation}. Let \( k > 0 \), and suppose that for some \( \sigma \in (0,1] \), we have
\[
    |B_r \cap \{w \geq k\}|_{d\mu} \geq \sigma |B_r|_{d\mu}
\]
with \( 0 < r < \frac{R}{16} \). Then, there exists a constant \( C = C(N, s) \) such that
\[
    |B_{6r} \cap \{w \leq 2\delta k\}|_{d\mu} \leq \frac{C}{\sigma \log\left( \frac{1}{2\delta} \right)} |B_{6r}|_{d\mu}
\]
for all \( \delta \in \left(0, \frac{1}{4}\right) \).
\end{lemma}
\begin{proof}
Choose a cut-off function \( \varphi \in C_c^{\infty}(\rnn)\) such that 
\[
    \text{Supp}(\varphi) \subset B_{7r}, \quad 0 \leq \varphi \leq 1 \quad \text{and} \quad \varphi \equiv 1 \text{ in } B_{6r}, \quad |\nabla \varphi| \leq \frac{C}{r}.
\]
Choosing test function \( \eta = w^{-1} \varphi^2 \) in \eqref{eq:transformed_equation} and Young's inequality give
\begin{align*}
    0 \leq & \int_{\om} \grad w \nabla \left( w^{-1} \varphi^2 \right) d\mu + \varepsilon C_0 \int_{\om} \frac{\varphi^2}{|x|^{2s}} d\mu   +\varepsilon \iint_{\rtwon} \left( w(x) - w(y) \right) \left(  \frac{\varphi^2(x)}{w(x)} - \frac{\varphi^2(y)}{w(y)} \right) d\nu \\
    \leq & -\frac{1}{2} \int_{\om}  \varphi^2 w^{-2} |\nabla w|^2\, d\mu + 2 \int_{\om} |\nabla \varphi|^2 \,d\mu + \varepsilon C_0 \int_{\om} \frac{\varphi^2}{|x|^{2s}} d\mu   +\varepsilon I_3 \\
    \leq & -\frac{1}{2} \int_{\om}  \varphi^2 w^{-2} |\nabla w|^2\, d\mu + C  |B_{6r}|_{d\mu} r^{-2} + C |B_{6r}|_{d\mu} r^{-2s} + I_3.
\end{align*}
For \( I_3 \), we break the integral as follows
\begin{align*}
    I_3 = \iint_{B_{8r} \times B_{8r}} d\nu + \iint_{B_{8r} \times \mathbb{R}^N \setminus B_{8r}} d\nu + \iint_{\mathbb{R}^N \setminus B_{8r} \times B_{8r}} d\nu = I_{3,1} + I_{3,2} + I_{3,3}.
\end{align*}
Using the same idea of \cite[Lemma $1.3$]{palatucci_castro2016local}, we obtain
\begin{align*}
    I_{3,1} & \leq -\frac{1}{C} \iint_{B_{8r} \times B_{8r}} \left| \log \left(\frac{w(x)}{w(y)}\right) \right|^2 \varphi^2(y) d\nu + C \iint_{B_{8r} \times B_{8r}} \left| \varphi(x) - \varphi(y) \right|^2 d\nu \\
    & \leq -\frac{1}{C} \iint_{B_{6r} \times B_{6r}} \left| \log \left(\frac{w(x)}{w(y)}\right) \right|^2 d\nu + C \iint_{B_{8r} \times B_{8r}} \left| \varphi(x) - \varphi(y) \right|^2 d\nu \\
    & \leq -\frac{1}{C} \iint_{B_{6r} \times B_{6r}} \left| \log \left(\frac{w(x)}{w(y)}\right) \right|^2 d\nu + C \lv \grad \varphi \rv_{L^{\infty}}^2 \iint_{B_{8r} \times B_{8r}} \frac{|x-y|^{2-N-2s}}{|x|^{d}|y|^{d}} \, dy \, dx \\
    & \leq -\frac{1}{C} \iint_{B_{6r} \times B_{6r}} \left| \log \left(\frac{w(x)}{w(y)}\right) \right|^2 d\nu + \frac{C}{r^2} \left( \int_{B_{8r}} |x|^{-2d N/(N+2-2s)}dx\right)^{(N+2-2s)/N} \\
    & \leq -\frac{1}{C} \iint_{B_{6r} \times B_{6r}} \left| \log \left(\frac{w(x)}{w(y)}\right) \right|^2 d\nu  +  C |B_{6r}|_{d\mu} r^{-2s},
\end{align*}
where the inequality in the penultimate line follows from the HLS inequality \cite{liebloss2001analysis}. Now
\begin{align*}
I_{3,2} = I_{3,3} &= \iint_{ B_{8r} \times \mathbb{R}^N \setminus B_{8r} } (w(x) - w(y)) w^{-1}(x) \varphi^2(x) \, d\nu \\
& \leq \iint_{B_{7r} \times \mathbb{R}^N \setminus B_{8r}} \frac{1}{|x - y|^{N + 2s}} \frac{1}{|x|^{d}}  \frac{1}{|y|^{d}} \, dx \, dy \\
&\leq \int_{B_{7r}} \frac{1}{|x|^{d}} \, dx \int_{\mathbb{R}^N \setminus B_{8r}} \frac{1}{|y|^{d}}  \frac{1}{( |y|-7r)^{N + 2s}} \, dy\\
& \leq C |B_{6r}|_{d\mu} r^{-2s}.
\end{align*}
Combining all estimates and the fact that \( r^{-2s} < r^{-2} \) for $r < 1$, we deduce that
\begin{align*}
\int_{B_{6r}} |\nabla(\log w)|^2 \, d\mu  &\leq \int_{\Omega}  \frac{\varphi^2}{w^2} |\nabla w|^2 \, d\mu +  \iint_{B_{6r} \times B_{6r}} \left| \log \left( \frac{w(x)}{w(y)} \right) \right|^2 d\nu \\
&\leq C \left[ |B_{6r}|_{d\mu} r^{-2} +  |B_{6r}|_{d\mu} r^{-2s} \right] \leq C \, r^{-2} \, |B_{6r}|_{d\mu}.
\end{align*}
For any \( \delta \in (0, \tfrac{1}{4}) \), define $v = \left[ \min \left\{ \log \left( \frac{1}{2\delta} \right),~ \log \left( \frac{k}{w} \right) \right\} \right]_+$ and
denote $(v)_{B_{6r}} = \frac{1}{|B_{6_r}|_{d\mu}} \int_{B_{6r}}v \, dx$. Then Lemma \ref{lem:weighted_poincare_inequality} yields
\begin{align*}
\int_{B_{6r}} |v(x) - (v)_{B_{6r}}| \, d\mu & \leq \left( \int_{B_{6r}} |v(x) - (v)_{B_{6r}}|^2 \, d\mu \right)^{1/2}   |B_{6r}|_{d\mu}^{1/2} \\
&\leq  C \, r \left( \int_{B_{6r}} |\nabla v|^2 \, d\mu \right)^{1/2}   |B_{6r}|_{d\mu}^{1/2} \\
& \leq C \, r \left( \int_{B_{6r}} |\nabla \log w|^2 \, d\mu \right)^{1/2}   |B_{6r}|_{d\mu}^{1/2} \\
& \leq  C \, |B_{6r}|_{d\mu}.
\end{align*}
Now, we know $ \{ v = 0 \} =  \{ w \geq k \}$ then $    |B_{6r} \cap \{ v = 0 \}|_{d\mu} \geq \frac{\sigma}{6^{N+2d}} |B_{6r}|_{d\mu}$.  Moreover,

\small{\begin{align*}
    \log \left( \frac{1}{2\delta} \right) = \frac{1}{|B_{6r} \cap \{ v = 0 \}|} \int_{B_{6r} \cap \{ v = 0 \}} \left( \log \left( \frac{1}{2\delta} \right) - v(x) \right) \, dx \leq \frac{6^{N+2d}}{\sigma} \left[ \log \left( \frac{1}{2\delta} \right) - (v)_{B_{6r}} \right].
\end{align*}}

Integrating the above inequality with respect to the measure $\mu$, we have
\begin{equation*}
    \left|\{ v = \log \left( \tfrac{1}{2\delta} \right) \} \cap B_{6r}\right|_{d\mu}   \log \left( \tfrac{1}{2\delta} \right)
    \leq \frac{6^{N+2d}}{\sigma} \int_{B_{6r}} |v(x) - (v)_{B_{6r}}| \, d\mu
    \leq \frac{C}{\sigma} |B_{6r}|_{d\mu}.
\end{equation*}
Therefore, for all $\delta \in \left( 0, \tfrac{1}{4} \right)$, we obtain the estimate
\begin{equation*} 
    |B_{6r} \cap \{ w \leq 2\delta k \}|_{d\mu} 
    \leq \frac{C}{\sigma \log \left( \frac{1}{2\delta} \right)} |B_{6r}|_{d\mu}. \qedhere
\end{equation*}
\end{proof}
\medskip

\begin{lemma}\label{lem:lower_bound_Harnack}
Assuming the hypothesis of Lemma \ref{lem:propagation_of_positivity}, there exists $\delta \in (0, \tfrac{1}{4})$ such that
\begin{equation*} 
    \inf_{B_{4r}} w \geq \delta k.
\end{equation*}
\end{lemma}
\begin{proof}
Take a smooth function $\varphi \in C_c^{\infty}(\om)$ with $\text{supp}(\varphi) \subseteq B_{\rho}$ with $\rho \in [r,6r]$. Testing the equation with the function $\eta = v_{\ell} \varphi^2$, where $v_{\ell} = (\ell - w)_+$. Then taking $\Omega_\ell = \Omega \cap \{x :  w(x) < \ell \}$ and using Young's inequality, we obtain
\begin{align*}
    0 &\leq \int_{\Omega} \nabla w   \nabla \eta  \, d\mu +\varepsilon C_0 \int_{\Omega} \frac{w \eta}{|x|^{2s}} \, d\mu 
    + \varepsilon\iint_{\mathbb{R}^{2N}} (w(x) - w(y)) (\eta(x) - \eta(y)) \, d\nu \\
    & \leq  -\frac{1}{2} \int_{\om_{\ell}} |\nabla v_{\ell}|^2 \varphi^2 \, d\mu 
+ 2 \int_{\om_{\ell}} |\nabla \varphi|^2 v_{\ell}^2 \, d\mu  -C_0 \int_{\om_{\ell}} \frac{v_{\ell}^2 \varphi^2}{|x|^{2s}} \, d\mu 
+ C_0 \int_{\om_{\ell}} \frac{\ell v_{\ell} \varphi^2}{|x|^{2s}} \, d\mu + \varepsilon I_3 \\
& \leq  -\frac{1}{2} \int_{\om_{\ell}} |\nabla v_{\ell}|^2 \varphi^2 \, d\mu 
+ 2 \int_{\om_{\ell}} |\nabla \varphi|^2 v_{\ell}^2 \, d\mu -\frac{C_0}{2} \int_{\om_{\ell}} \frac{v_{\ell}^2 \varphi^2}{|x|^{2s}} \, d\mu 
+ \frac{C_0}{2} \int_{\om_{\ell}} \frac{\ell^2 \varphi^2}{|x|^{2s}} \, d\mu + I_3.
\end{align*}
For $I_3$, we proceed as in \cite[Lemma $3.2$]{castro2014nonlocal_harnack}, leading to
\begin{equation*} 
\begin{aligned}
2 I_3 &\leq  \iint_{B_\rho \times B_\rho} \left( - \left| v_{\ell}(x) \varphi(x) - v_{\ell}(y) \varphi(y) \right|^2 + C \left( \max\{ v_{\ell}(x), v_{\ell}(y) \} \right)^2 \left| \varphi(x) - \varphi(y) \right|^2  \right)\, d\nu \\
&\quad + C \left( \sup_{x \in \text{supp}(\varphi)} \int_{\mathbb{R}^N \setminus B_{\rho}} \frac{v_{\ell}(y)}{|x - y|^{N+2s}} \, \frac{dy}{|y|^{d}} \right)\int_{B_{\rho}} \frac{v_{\ell}(x) \varphi(x)^2}{|x|^{d}} \, dx .
\end{aligned}
\end{equation*}
Combining all, we obtain
\begin{align}
  &\int_{\om_{\ell}} |\nabla v_{\ell}|^2 \varphi^2 \, d\mu +  C_0 \int_{\om_{\ell}} \frac{v_{\ell}^2 \varphi^2}{|x|^{2s}} \, d\mu + \iint_{B_\rho \times B_\rho} \left| v_{\ell}(x) \varphi(x) - v_{\ell}(y) \varphi(y) \right|^2\, d\nu \notag\\
  &\leq  2 \int_{\om_{\ell}} |\nabla \varphi|^2 v_{\ell}^2 \, d\mu 
+ \frac{C_0}{2} \int_{\om_{\ell}} \frac{\ell^2 \varphi^2}{|x|^{2s}} \, d\mu +
C \iint_{B_\rho \times B_\rho}  \left( \max\{ v_{\ell}(x), v_{\ell}(y) \} \right)^2 \left| \varphi(x) - \varphi(y) \right|^2  d\nu \notag\\
&\quad + C \left( \sup_{x \in \text{supp}(\varphi)} \int_{\mathbb{R}^N \setminus B_{\rho}} \frac{v_{\ell}(y)}{|x - y|^{N+2s}} \, \frac{dy}{|y|^{d}} \right)\int_{B_{\rho}} \frac{v_{\ell}(x) \varphi(x)^2}{|x|^{d}} \, dx \notag\\
& \leq C(K_1 + K_2 + K_3 + K_4). \label{eq:lower_bd_harnack_eq1}
\end{align}
Now, we recast the inequality \eqref{eq:lower_bd_harnack_eq1} into a format suitable for applying a iteration lemma. To this end, let
\begin{equation*}
\begin{aligned}
\ell \equiv \ell_j &= \delta k + 2^{-j} \delta k, \qquad
\rho = \rho_j :=4r +  2^{1-j}r,& \qquad 
\tilde{\rho}_j := \frac{\rho_{j+1} + \rho_j}{2}.
\end{aligned}
\end{equation*}
We have $\ell_j - \ell_{j+1} = 2^{-j-1} \delta k$ which implies $\ell_j - \ell_{j+1} \geq 2^{-j-2} \ell_j$. Define the function
\begin{equation}\label{eq:vj_lower_bound}
v_{\ell} \equiv v_j := (\ell_j - w)_+ \geq (\ell_j - \ell_{j+1}) \chi_{\{ w < \ell_{j+1} \}} \geq 2^{-j-2} \ell_j \chi_{\{ w < \ell_{j+1} \}}.
\end{equation}
Define the balls $B_j$ and cut-off functions $\varphi_j$ for all \( j = 0, 1, \dots \) as
\begin{equation*}
B_j := B_{\rho_j}(0), \quad 
\varphi_j \in C_c^\infty(B_{\tilde{\rho}_j}), 
\quad \text{with} \quad 
0 \leq \varphi_j \leq 1, \quad 
\varphi_j \equiv 1 \text{ in } B_{{j+1}}, \quad 
|\nabla \varphi_j| \leq C\frac{2^{j}}{r},
\end{equation*}
and denote $\widetilde{B}_j = B_j \cap \{ w < \ell_j \}$.
Now, estimate $K_1$ in terms of $\ell_j$ as
\begin{equation*}
K_1 = \int_{B_j} |\nabla \varphi_j|^2 v_j^2 \, d\mu 
\leq C \ell_j^2   \frac{2^{2j}}{r^2} 
\left| B_j \cap \{ w < \ell_j \} \right|_{d\mu} = C \ell_j^2   \frac{2^{2j}}{r^2} \widetilde{B}_j.
\end{equation*}
Next, we estimate the Hardy term by noting that
\begin{align*}
\int_{\widetilde{B}_j} \frac{1}{|x|^{2s}} \, d\mu 
&= \frac{1}{|B_j|} \int_{B_j} \left( \int_{\widetilde{B}_j} 
\frac{1}{|x |^{2s + 2d}} \, dx \right) dy = \frac{1}{|B_j|} (J_1 + J_2 + J_3),
\end{align*}
where $J_i = \iint_{R_i \cap (\widetilde{B_j} \times B_j)} \frac{1}{|x|^{2s + 2d}} dx dy$ for $i = 1,2,3$ with regions are defined as follows
\begin{equation}\label{eq:region_breaking}
R_1 = \left\{ \frac{1}{2} \leq  \frac{|x|}{|y|} < 2 \right\}, \quad
R_2 = \left\{ |x| < \frac{1}{2}|y| \right\}, \quad
R_3 = \left\{ |y| \leq \frac{1}{2}|x| \right\}.
\end{equation}
For the first term
\begin{equation*}
J_1 \leq C \iint_{R_1 \cap (\widetilde{B_j} \times B_j)} \frac{1}{|y|^{2s}} 
  \frac{1}{|x|^{2d}} \, dx \, dy 
\leq C \left| B_j \cap \{ w < \ell_j \} \right|_{d\mu}   r^{N - 2s}.
\end{equation*}
Similarly, for the third term
\begin{equation*}
J_3 \leq C \left| B_j \cap \{ w < \ell_j \} \right|_{d\mu}   r^{N - 2s}.
\end{equation*}
Finally, for $J_2$, Lemma \ref{lem:stein_strong_lemma} with parameters $\delta = -2s + \frac{N}{2}$ and $p=2$ gives
\begin{align*} 
J_2 
&\leq C \int_{\widetilde{B}_j} \frac{1}{|x|^{2 s  + 2 d - N + \delta}}   |x|^{-N + \delta} 
\left( \int_{|y| \leq |x|} |y|^{-\delta}   \frac{1}{|y|^{-\delta}} \, dy \right) dx\\
&\leq C \int_{\widetilde{B}_j} \frac{1}{|x|^{2 s  + 2 d - N + \delta}}   |x|^{-N/2}
\left( \int_{B_j} \frac{1}{|y|^{-2\delta}} \, dy \right)^{1/2} dx\\
&\leq C   r^{N - 2s}   \left| B_j \cap \{ w < \ell_j \} \right|_{d\mu}.
\end{align*}
Thus we obtain that
\begin{equation*}
K_2
\leq C   r^{ - 2s}   \left| B_j \cap \{ w < \ell_j \} \right|_{d\mu}.
\end{equation*}
For the fractional term $K_3$, we have
\begin{align*} 
K_3 &
\leq  \ell_j^2 \iint_{B_j \times \widetilde{B}_j} 
\frac{\| D\varphi_j \|^2_\infty}{|x - y|^{N + 2s-2}} 
\frac{dy}{|y|^{d}}   \frac{dx}{|x|^{d}}
\leq C \ell_j^2 2^{2j} r^{-2} 
\iint_{B_j \times \widetilde{B}_j} 
\frac{1}{|x - y|^{N + 2s-2}} 
  \frac{dy}{|y|^{d}}   \frac{dx}{|x|^{d}} \\
&= C \ell_j^2 2^{2j} r^{-2} \left( L_1 + L_3 + L_2 \right),
\end{align*}
where $L_i = \iint_{R_i \cap [B_j \times (B_j \cap \{ w < \ell_j \})]}$ with regions $R_1, R_2$ and $R_3$ are same as in \eqref{eq:region_breaking}.

We now estimate each region. For \( L_1 \), we obtain
\begin{equation*}
L_1 \leq C \iint_{R_1 \cap [B_j \times (B_j \cap \{ w < \ell_j \})]} \frac{1}{|x - y|^{N + 2s - 2}}   \frac{dy}{|y|^{2d}}   dx
\leq C | B_j \cap \{ w < \ell_j \} |_{d\mu}   r^{2 - 2s}.
\end{equation*}
Similarly, for \( L_2 \), we have
\begin{equation*}
L_2 \leq C |B_j \cap \{ w < \ell_j \} |_{d\mu}   r^{2 - 2s}.
\end{equation*}
For \( L_3 \), we invoke the Lemma \ref{lem:stein_strong_lemma} with parameters \( q \) and \( \de \) satisfying the relations $q = \frac{N}{N+d - \delta}$, $ \delta < \frac{N(q-1)}{q}$, and $q > \frac{N}{N-d}$ and noting that in $R_2$ we have $|x-y| \geq |y| - |x| \geq |x|$, we estimate
\begin{align*} 
L_3 &\leq C \int_{B_j \cap \{ w < \ell_j \}} \frac{1}{|y|^{d}} 
\left( \int_{\{ 2|x| \leq |y| \}} \frac{1}{|x|^{d}}   \frac{1}{|x |^{N + 2s - 2}} \, dx \right) dy\\
& \leq C \int_{\widetilde{B}_j} \frac{1}{|y|^{d - N + \delta}} |y|^{-N+\delta}
\left( \int_{ \{|x| \leq |y|\} } \frac{|x|^{-\delta}}{|x |^{- \delta + d+ N + 2s - 2}} \, dx \right) dy\\
& \leq C \int_{\widetilde{B}_j} \frac{1}{|y|^{2d}} 
\left( \int_{B_j} \left( \frac{1}{|x|^{- \delta + d + N + 2s - 2}} \right)^q \, dx \right)^{1/q} dy\\
& \leq C \left| B_j \cap \{ w < \ell_j \} \right|_{d\mu}   r^{2 - 2s}.
\end{align*}
Therefore, we have 
\begin{equation*}
    K_3 \leq C \ell_j^2   2^{2j}   r^{-2s} |\widetilde{B}_j|_{d\mu}.
\end{equation*}
We now estimate $K_4$,
\begin{align*}
K_4 
\leq C \ell_j^2 \int_{\widetilde{B}_j} \int_{\rnn \setminus B_j} 
\frac{2^{j(N+2s)}}{|y|^{N+2s}}   
\frac{dy}{|y|^{d}}   \frac{dx}{|x|^{d}}
\leq C \ell_j^2   2^{j(N+2s)}   r^{-2s} \left| \widetilde{B}_j \right|_{d\mu}.
\end{align*}

Using values obtained of $K_i$ in \eqref{eq:lower_bd_harnack_eq1} and \eqref{eq:vj_lower_bound}, we obtain the following estimate
\begin{align*}
 (\ell_j -  \ell_{j+1})^2 
&\left( \frac{| B_{j+1} \cap \{ w < \ell_{j+1}  \}|_{d\mu}}{|B_{j+1}|_{d\mu}}   \right)^{\frac{2}{2^*}} \leq \left( \frac{1}{|B_{j+1}|_{d\mu}} \int_{B_{j+1}} (v_j \varphi_j)^{2^*} \, d\mu \right)^{\frac{2}{2^*}} \\
&   \leq  C\left( \frac{1}{|B_{j}|_{d\mu}} \int_{B_{j}} (v_j \varphi_j)^{2^*} \, d\mu \right)^{\frac{2}{2^*}} \\
& \leq C \left( \frac{r^2}{|B_j|_{d\mu}} \int_{B_j} |\nabla(v_j \varphi_j)|^2 \, d\mu \right)\\
&\leq \frac{C r^2}{|B_j|_{d\mu}} \left[ 
\int_{B_j} |\nabla(v_j \varphi_j)|^2 \, d\mu 
+   \int_{\widetilde{B}_j} \frac{(v_j \varphi_j)^2}{|x|^{2s}} \, d\mu 
\right. \\
& \quad \left.+ \iint_{B_j \times B_j} \frac{|v_j(x) \varphi_j(x) - v_j(y) \varphi_j(y)|^2}
{|x - y|^{N+2s}} \, \frac{dy}{|y|^{d}} \frac{dx}{|x|^{d}}
\right]\\
& \leq \frac{C r^2}{|B_j|_{d\mu}} \left[ 
\ell_j^2   \frac{2^{2j}}{r^2}   |\widetilde{B}_j|_{d\mu} 
+  \ell_j^2   \frac{|\widetilde{B}_j|_{d\mu}}{r^{2s}} 
+  \ell_j^2   2^{j(N+2s)}   \frac{|\widetilde{B}_j|_{d\mu}}{  r^{2s}} 
\right]\\
& \leq C   2^{j(N+2s + 2)}   \ell_j^2   
\frac{\left| \widetilde{B}_j  \right|_{d\mu}}{|B_j|_{d\mu}},
\end{align*}
where $C$ is a constant independent of $\varepsilon$.
Now, denote $A_j = \frac{\left| \widetilde{B}_j  \right|_{d\mu}}{|B_j|_{d\mu}}$, then we obtain
\begin{equation*}
A_{j+1}^{\frac{2}{2^*}} \leq C   \frac{\ell_j^2   2^{j(N+2s + 2)}}{(\ell_j - \ell_{j+1})^2}   A_j \leq C 2^{j(N+2s + 4)}   A_j.
\end{equation*}
Following the iteration scheme(\cite[Lemma $7.1$]{giusti2003iteration_lemma}), we conclude $\lim_{j \to \infty} A_j = 0$. Using this completes the proof.
\end{proof}
Next, we derive a reverse H\"{o}lder inequality for the solutions of \eqref{eq:transformed_equation}.

\begin{lemma}[Reverse H\"{o}lder Inequality]
Let \( w \) be a supersolution to \eqref{eq:transformed_equation}. Then, for all \( 0 < \gamma_1 < \gamma_2 < \frac{N}{N - 2} \), we have:
\begin{equation} \label{lem:reverse_holder}
\left( \frac{1}{|B_r|_{d\mu}} \int_{B_r} w^{\gamma_2} \, d\mu \right)^{1/\gamma_2}
\leq C \left( \frac{1}{|B_{3r/2}|_{d\mu}} \int_{B_{3r/2}} w^{\gamma_1} \, d\mu \right)^{1/\gamma_1}.
\end{equation}
\end{lemma}
\begin{proof}
Let \( q \in (1,2) \) and \( n \in \ntrl\). Define \( \widetilde{w} = w + \frac{1}{n} \), and let \( \psi \) be a cut-off function with \( \operatorname{Supp}(\psi) \subseteq B_{\tau r} \), such that
\begin{equation*}
    \psi \equiv 1 \text{ in } B_{\tau' r} \quad \text{and} \quad |\nabla \psi| \leq \frac{C}{(\tau - \tau')r}, \quad \text{where } \frac{1}{2} \leq \tau' < \tau < \frac{3}{2}.
\end{equation*}
Taking the test function \( \eta = \widetilde{w}^{1-q} \psi^2 \), we deduce the following inequality
\begin{equation}\label{eq:rev_holder_eq1}
\begin{aligned}
0 \leq 
& \int_{B_{\tau r}} \nabla w  \nabla\left( \widetilde{w}^{1-q} \psi^2 \right) \, d\mu 
+ \varepsilon C_0 \int_{B_{\tau r}} \frac{w\widetilde{w}^{1-q} \psi^2}{|x|^{2s}} \, d\mu \\
&+ \varepsilon \iint_{\mathbb{R}^{2N}} \left( \widetilde{w}(x) - \widetilde{w}(y) \right) 
\left( \frac{\psi^2(x)}{\widetilde{w}^{q-1}(x)} - \frac{\psi^2(y)}{\widetilde{w}^{q-1}(y)} \right) \, d\nu\\
= &~ I_1 + \varepsilon C_0 I_2 + \varepsilon I_3 \leq I_1 + C_0 I_2 + I_3.
\end{aligned}
\end{equation}
To estimate \( I_1 \), we apply Young's inequality with \( \epsilon = \frac{(q - 1)}{2} \), yielding
\begin{equation}
\begin{aligned}
I_1 
&= - (q - 1) \int_{B_{\tau r}} |\nabla w|^2 \, \widetilde{w}^{-q} \psi^2 \, d\mu 
+ \int_{B_{\tau r}} \nabla w  \nabla \psi \, \widetilde{w}^{1-q} \psi \, d\mu \\
&\leq -\frac{(q - 1)}{2} \int_{B_{\tau r}} |\nabla w|^2 \, \widetilde{w}^{-q} \psi^2 \, d\mu 
+ \frac{4}{(q - 1)} \int_{B_{\tau r}} |\nabla \psi|^2 \, \widetilde{w}^{2 - q} \, d\mu \\
&\leq -\frac{(q - 1)}{2} \int_{B_{\tau r}} |\nabla w|^2 \, \widetilde{w}^{-q} \psi^2 \, d\mu 
+ \frac{C}{(\tau - \tau')^2 r^2} \int_{B_{\tau r}} \widetilde{w}^{2 - q} \, d\mu.
\end{aligned}
\label{eq:I1_estimate}
\end{equation}
For \( I_2 \), breaking the region into subregions as in \eqref{eq:region_breaking}, and estimating, we obtain
\begin{equation}
    I_2 \leq C \, r^{-2s} \int_{B_{\tau r}} w \widetilde{w}^{1 - q} \, d\mu.
    \label{eq:I2_estimate}
\end{equation}
For \( I_3 \), we follow the approach used in \cite[Lemma $3.7$]{peral2016CZ}, to get
\begin{equation}
\begin{aligned}
    I_3 \leq & -C \iint_{B_{\tau r} \times B_{\tau r}} 
    \left( \frac{\psi}{\widetilde{w}^{(q - 2)/2}}(x) - \frac{\psi}{\widetilde{w}^{(q - 2)/2}}(y) \right)^2 \, d\nu + C \, \frac{r^{-2s}}{(\tau - \tau')^2} \int_{B_{\tau r}} \widetilde{w}^{2 - q} \, d\mu.
\end{aligned}
\label{eq:I3_estimate}
\end{equation}

Using estimates \eqref{eq:I1_estimate}, \eqref{eq:I2_estimate}, and \eqref{eq:I3_estimate} in \eqref{eq:rev_holder_eq1}, and $r^{-2s} \leq r^{-s}$ for $r \leq 1$, we deduce the inequality
\small{\begin{equation*}
\begin{aligned}
    \int_{B_{\tau r}} \left| \nabla\left( \widetilde{w}^{(2 - q)/2} \right) \right|^2 \psi^2 \, d\mu + \iint_{B_{\tau r} \times B_{\tau r}} 
    \left( \frac{\psi}{\widetilde{w}^{(q - 2)/2}}(x) - \frac{\psi}{\widetilde{w}^{(q - 2)/2}}(y) \right)^2 \, d\nu \leq \frac{C r^{-2}}{(\tau - \tau')^2} \int_{B_{\tau r}} \widetilde{w}^{2 - q} \, d\mu.
\end{aligned}
\end{equation*}}

Now, from the above and applying the weighted Sobolev inequality in Lemma \ref{lem:weighted_sobolev_inequality}, we deduce
\begin{equation*}
\begin{aligned}
\left( \frac{1}{|B_{\tau' r}|_{d\mu}} \int_{B_{\tau' r}} \widetilde{w}^{(2 - q)2^*/2 } \, d\mu \right)^{\frac{N - 2}{N}}
&\leq C \left( \frac{1}{|B_{\tau r}|_{d\mu}} \int_{B_{\tau r}} \left( \widetilde{w}^{(2 - q)/2} \psi \right)^{2^*} \, d\mu \right)^{\frac{N - 2}{N}}\\
&\leq \frac{C \tau^2 r^2}{|B_{\tau r}|_{d\mu}} \int_{B_{\tau r}} 
\left| \nabla\left( \widetilde{w}^{(2 - q)/2} \psi \right) \right|^2 \, d\mu \\
&\leq \frac{C}{|B_{\tau r}|_{d\mu} (\tau - \tau')^2} \int_{B_{\tau r}} \widetilde{w}^{2 - q} \, d\mu.
\end{aligned}
\end{equation*}
Applying the Monotone Convergence Theorem and letting \( n \to \infty \). Since \( 1 < q < 2 \) is arbitrary, H\"{o}lder's inequality implies the desired estimate.
\end{proof}

The next two lemmas follows directly from \cite[Lemma $4.10$ and $4.11$]{dipierro2016fractionalwithHardy}.
\begin{lemma} \label{lemma:measure_enlargement}
Assume that \( E \subset B_r(x_0) \) is a measurable set. For \( \bar{\delta} \in (0,1) \), we define the \emph{enlargement} of \( E \) as
\begin{equation*} 
[E]_{\bar{\delta}} := \bigcup_{\rho > 0} \left\{ B_{3\rho}(x) \cap B_r(x_0) \,:\, x \in B_r(x_0),\; |E \cap B_{3\rho}(x)|_{d\mu} > \bar{\delta} |B_{\rho}(x)|_{d\mu} \right\}.
\end{equation*}
Then, there exists a constant \( \tilde{C} \), depending only on \( N \), such that one of the following holds
\begin{enumerate}
    \item \( |[E]_{\bar{\delta}}|_{d\mu} \geq \dfrac{\tilde{C}}{\bar{\delta}} |E|_{d\mu}, \) or
    \item \( [E]_{\bar{\delta}} = B_r(x_0). \)
\end{enumerate}
\end{lemma}
\begin{lemma}\label{lemma:main_supersolution}
Assume that \( w \) is a nonnegative supersolution to \eqref{eq:transformed_equation}. Then, there exists \( \eta \in (0,1) \), depending only on \( N \), such that the following inequality holds
\begin{equation*} 
\left( \frac{1}{|B_r|_{d\mu}} \int_{B_r} w^{\eta} \, d\mu(x) \right)^{1/\eta} \leq C \inf_{B_r} w.
\end{equation*}
\end{lemma}
Note that the constant in the above lemma is independent of $\varepsilon$, as it follows from Lemmas~\ref{lem:lower_bound_Harnack} and~\ref{lemma:measure_enlargement}, whose constants are also independent of $\varepsilon$. Consequently, the constant in the subsequent proof of the Harnack inequality is likewise independent of $\varepsilon$.
\begin{proof}[Proof of Theorem \ref{thm:harnack_inequality_transformed}(Harnack Inequality):]
Applying Lemma \ref{lem:reverse_holder} with $\gamma_1 = \eta$ and $\gamma_2 = q$ and Lemma \ref{lemma:main_supersolution}, we get 
\begin{equation*}
    \left( \frac{1}{|B_r|_{d\mu}} \int_{B_r} w^{q} \, d\mu(x) \right)^{1/q} \leq C  \left( \frac{1}{|B_{3r/2}|_{d\mu}} \int_{B_{3r/2}} w^{\eta} \, d\mu(x) \right)^{1/\eta} \leq C \inf_{B_{3r/2}} u.
\end{equation*}
This concludes the proof. 
\end{proof}


Next, we obtain a uniform estimate for the solutions of \eqref{eq:transformed_equation} using the Moser iteration technique. For this we require the following two lemmas

\begin{lemma}\label{lem:moser_iter_lem1}
 If $w$ satisfies 
 \begin{equation*}
      -\mathrm{div}\left( \frac{\nabla w(x)}{|x|^{2 d}}\right) + \varepsilon C_0 \frac{w}{|x|^{2s + 2d}} + \varepsilon (-\Delta_{d})^s w  = F(x,w(x))
 \end{equation*}
  where $F \in L^q(\om)$ for $q > N/2$, then $w \in L^{\infty}(\om)$.   
\end{lemma}
\begin{proof}
For \( 0 < \varepsilon < 1 \), consider the regularized function $f_\varepsilon(t) = \left( \varepsilon^2 + t^2 \right)^{1/2}$. Now, take the test function
\begin{equation*}
    \varphi = \psi \, f_\varepsilon'(w),
\end{equation*}
where \( \psi \in C_c^\infty(\Omega) \) is a positive function.
Then, using the weak formulation, we obtain
\begin{equation*}
\begin{aligned}
    \int_\Omega \frac{\nabla w  \nabla \varphi}{|x|^{2d}} \, dx 
    + \varepsilon C_0 \int_\Omega \frac{w \varphi}{|x|^{2s + 2d}} \, dx 
    + \varepsilon \iint_{\rtwon} (w(x) - w(y))(\varphi(x) - \varphi(y)) \, d\nu 
    \leq \int_\Omega |F| \left| f_\varepsilon'(w) \right| |\psi| \, dx.
\end{aligned}
\end{equation*}
By using Lemma \ref{ineq:A1} and invoking Fatou's Lemma, we deduce the inequality
\begin{equation*}
\begin{aligned}
    \int_\Omega \frac{\nabla |w| \,  \nabla \psi}{|x|^{2d}} \, dx 
    + \varepsilon \iint_{\mathbb{R}^{2N}} \left( |w(x)| - |w(y)| \right) ( \psi(x) - \psi(y) ) \, d\nu 
    \leq \int_\Omega |F| \, \psi \, dx,
\end{aligned}
\end{equation*}
for all positive test functions \( \psi \in C_c^\infty(\Omega) \). Hence, by density, the above inequality holds for all $ \psi \in \dot{H}^{s,d}(\Omega) \cap \dot{H}^1(\Omega, |x|^{-2d})$. \\

Let us define \( w_M = \min\{ |w|, M \} \) for some \( M > 0 \). Let \( \beta > 0 \) and \( \delta > 0 \), and consider the test function
\begin{equation*}
    \psi = (w_M + \delta)^\beta - \delta^\beta.
\end{equation*}
Using Lemma \eqref{ineq:A3}, we obtain the following estimate
\begin{equation}
\begin{aligned}
    \frac{4\beta}{(\beta + 1)^2} \int_{\om}  | \grad  (w_M + \delta)^{\frac{\beta + 1}{2}}|^2 &\frac{dx}{|x|^{2d}} 
    \leq 
    \beta \int_\Omega (w_M + \delta)^{\beta - 1} |\nabla w_M|^2 \, \frac{dx}{|x|^{2d}} \\
    &\quad + \varepsilon \iint_{\mathbb{R}^{2N}} (|w(x)| - |w(y)|) \left[ (w_M(x) + \delta)^\beta - (w_M(y) + \delta)^\beta \right] \, d\nu \\
    &\leq \int_\Omega |F| (w_M + \delta)^\beta \, dx.
\end{aligned}
\label{eq:main_energy_estimate}
\end{equation}
Now, invoking the Sobolev inequality with weights \eqref{ineq:sobolev_inequality_weight} and Young's inequality
\begin{equation*}
\begin{aligned}
S\int_{\om}  | \grad  (w_M + \delta)^{\frac{\beta + 1}{2}}|^2 \frac{dx}{|x|^{2d}} 
&\geq 
\left( \int_\Omega \left[ (w_M + \delta)^{\frac{\beta + 1}{2}} - \delta^{\frac{\beta + 1}{2}} \right]^{2^*} \frac{dx}{|x|^{ 2^* d}} \right)^{\frac{2}{2^*}}, \\
    &\geq  \left( \int_\Omega \left( \frac{(w_M + \delta)^{\beta + 1}}{2} - \delta^{{\beta + 1}} \right)^{\frac{2^*}{2}} \frac{dx}{|x|^{2^* d}} \right)^{\frac{2}{2^*}} \\
    &\geq \frac{1}{2} \left( \int_\Omega (w_M + \delta)^{\frac{(\beta + 1) 2^*}{2}} \frac{dx}{|x|^{2^* d}} \right)^{\frac{2}{2^*}} - \delta^{{\beta + 1}} \left( \int_\Omega \frac{dx}{|x|^{ 2^* d}} \right)^{\frac{2}{2^*}},
\end{aligned}
\end{equation*}
where \( S \) is the best constant in the Sobolev embedding.
Substituting the above estimate into inequality \eqref{eq:main_energy_estimate}, we get
\begin{equation}
\begin{aligned}
    \frac{\delta}{2} \left( \int_\Omega (w_M + \delta)^{\frac{\beta 2^*}{2}} \frac{dx}{|x|^{2^* d}} \right)^{\frac{2}{2^*}}
    &\leq \frac{1}{2} \left( \int_\Omega (w_M + \delta)^{\frac{(\beta + 1) 2^*}{2}} \frac{dx}{|x|^{2^* d}} \right)^{2/2^*} \\
    &\leq \delta^{\beta + 1} \left( \int_\Omega \frac{dx}{|x|^{2^* d}} \right)^{\frac{2}{2^*}} + \frac{(\beta + 1)^2}{4\beta S} \int_\Omega |F| (w_M + \delta)^\beta dx.
\end{aligned}
\label{eq:improved_estimate_main}
\end{equation}
Moreover, we also have the following inequality
\begin{equation}
\begin{aligned}
    \delta^\beta \left( \int_\Omega \frac{dx}{|x|^{2^* d}} \right)^{2/2^*}
    &\leq \frac{1}{\beta} \left( \frac{\beta + 1}{2} \right)^2 
    \left( \int_\Omega \frac{dx}{|x|^{2^* d}} \right)^{\frac{2}{2^*} - \frac{1}{q'}} 
    \left( \int_\Omega \frac{(w_M + \delta)^{\beta q'}}{|x|^{2^* d}} dx \right)^{1/q'}
\end{aligned}
\label{eq:holder_embedding_estimate}
\end{equation}
where the exponents satisfy the usual duality relation $\frac{1}{q} + \frac{1}{q'} = 1$.

Substituting the estimate \eqref{eq:holder_embedding_estimate} into \eqref{eq:improved_estimate_main}, we deduce the following refined inequality
\begin{equation*}
\begin{aligned}
    \left( \int_\Omega (w_M + \delta)^{\frac{\beta 2^*}{2}} \frac{dx}{|x|^{2^* d}} \right)^{\frac{2}{2^*}}
    &\leq C \left( \frac{\beta + 1}{2} \right)^2 \frac{1}{\beta}
    \left[ \left( \int_\Omega \frac{(w_M + \delta)^{\beta q'}}{|x|^{2^* d}} dx \right)^{\frac{1}{q'}} \right. \\
    &\quad \left. \left( \int_\Omega \frac{dx}{|x|^{2^* d}} \right)^{\frac{2}{2^*} - \frac{1}{q'}} + \frac{2}{\delta} \left\| F  |x|^{2^* d/q'} \right\|_{L^q} \right]
\end{aligned}
\end{equation*}
We now choose the parameter \( \delta \) to balance the terms as follows
\begin{equation*}
    \delta = 2 \left\| F  |x|^{2^* d/q'} \right\|_{L^q} 
    \left( \int_\Omega \frac{dx}{|x|^{2^* d}} \right)^{\frac{1}{q'} - \frac{2}{2^*}}
\end{equation*}
Letting \( \ell = \beta q' \) and $ \chi = \frac{N}{N - 2}  \frac{1}{q'} > 1$, we obtain
\begin{equation*}
    \left( \int_\Omega (w_M + \delta)^{\chi \ell} \frac{dx}{|x|^{2^* d}} \right)^{\frac{1}{\chi \ell}}
    \leq \left( C \left[ \int_\Omega \frac{dx}{|x|^{2^* d}} \right]^{\frac{2}{2^*} - \frac{1}{q'}} \right)^{\frac{q'}{\ell}}
    \left( \frac{q'}{\ell} \right)^{\frac{q'}{\ell}}
    \left( \frac{q' + \ell}{2\ell} \right)^{\frac{2}{\ell}}  \left( \int_\Omega \frac{(w_M + \delta)^{\ell}}{|x|^{2^* d}} dx \right)^{\frac{1}{\ell}}.
\end{equation*}
Now, applying the Moser-type iteration as detailed in \cite[Theorem $3.1$]{brasco_second_eigenvalue}, we obtain the desired result.
\end{proof}

The following lemma is motivated from Theorem $3.3$ of \cite{squassina_regularity}.
\begin{lemma}\label{lem:moser_iter_lem2}
If $w$ satisfies equation \eqref{eq:transformed_equation} then $w \in L^q(\om)$ for any $q \geq 1$.    
\end{lemma}
\begin{proof}
Let $g_\beta(t) = \text{sgn}(t)\,|t|\,|t_k|^\beta$, where $\beta > 0$,  and $t_k = \min\{t, k\}$. Also, define the auxiliary function
\begin{equation*}
G_\beta(t) = \int_0^t g_\beta'(t)^{1/2} \, dt \geq \frac{2(\beta + 1)^{1/2}}{\beta + 2} g_{\beta/2}(t).
\end{equation*}
Then, by Lemma \ref{ineq:A2}, we have
\begin{equation*}
\langle (-\Delta)^s_d w, g_\beta(w) \rangle_{s,d} \geq \left[ G_\beta(w) \right]_{s,d}^2.
\end{equation*}
Moreover, we have the following identity
\begin{equation*}
\int_\Omega \frac{w |w|\, |w_k|^\beta \, \text{sgn}(w)}{|x|^{2s + 2 d}}\, dx = \int_\Omega \frac{|w|^2\, |w_k|^\beta}{|x|^{2s + 2 d}}\, dx \geq 0.
\end{equation*}
Finally, considering the energy estimate with \( g_\beta(w) \), we compute
\begin{equation*}
\begin{aligned}
\int_\Omega \nabla w  \nabla \big( g_\beta(w) \big)\, d\mu 
&= \int_\Omega |\nabla w|^2 |w_k|^\beta\, d\mu + \beta \int_\Omega |w_k|^\beta |\nabla w_k|^2 \, d\mu \\
&\geq \left(1 + \frac{\beta}{4}\right)^{-1} \int_\Omega \left( |\nabla w|^2 |w_k|^\beta + \left( \frac{\beta^2}{4} + \beta \right) |w_k|^\beta |\nabla w_k|^2 \right) d\mu \\
&= \left(1 + \frac{\beta}{4} \right)^{-1} \int_\Omega |\nabla (w |w_k|^{\beta/2})|^2\, d\mu.
\end{aligned}
\end{equation*}
Combining all previous results and using \( g_\beta(w) \) as a test function along with Sobolev inequality, we deduce
\begin{equation}\label{eq:moser_estimate1}
\begin{aligned}
S \left( 1 + \frac{\beta}{4} \right)^{-1}
\left( \int_\Omega \frac{|w|^{2^*} |w_k|^{\beta \frac{2^*}{2}}}{|x|^{2^* d}} \, dx \right)^{\frac{2}{2^*}}
&\leq \int_\Omega \nabla w  \nabla (g_\beta(w)) \, d\mu + \varepsilon \left[ G_\beta(w) \right]_{s,d}^2 + \varepsilon C_0\int_\Omega \frac{|w|^2 |w_k|^\beta}{|x|^{2s + 2 d}}\, dx \\
&\leq \int_\Omega \frac{|w|^{2^*} |w_k|^\beta}{|x|^{2^* d}} \, dx 
+ \lambda \int_\Omega \frac{|w|^{q+1} |w_k|^\beta}{|x|^{(p+1)d}}\, dx \\
&= I_1 + I_2.
\end{aligned}
\end{equation}
We now estimate each term separately. First, we handle \( I_1 \)
\begin{equation}\label{eq:I_1_estimate_moser}
\begin{aligned}
I_1 &\leq \int_{\{|w| < K\}} \frac{|w|^{2^*} |w_k|^\beta}{|x|^{2^* d}} \, dx 
+ \int_{\{|w| \geq K\}} \frac{|w|^{2^*} |w_k|^\beta}{|x|^{2^*d}} \, dx \\
&\leq K^\beta \int_{\{|w| < K\}} \frac{|w|^{2^*}}{|x|^{2^* d}} \, dx 
+ \left( \int_{\{|w| > K\}} \frac{|w|^{2^*}}{|x|^{2^* d}} \, dx \right)^{1 - \frac{2}{2^*}} 
\left( \int_{\{|w| > K\}} \frac{|w|^{2^*}|w_k|^{\beta 2^*/2}}{|x|^{2^* d}} \, dx \right)^{\frac{2}{2^*}}.
\end{aligned}
\end{equation}

Similarly, for \( I_2 \), we obtain
\begin{equation}\label{eq:I_2_estimate_moser}
\begin{aligned}
I_2 &\leq K^\beta \int_{\{|w| < K\}} \frac{|w|^{p + 1}}{|x|^{(p+1)d}} \, dx 
+ \left( \int_{\{|w| > K\}} \frac{|w_k|^{\beta 2^*/2}}{|x|^{2^* d}} \, dx \right)^{2/2^*}
\left( \int_{\{|w| > K\}} \left(\frac{|w|^{q + 1}}{|x|^{(q-1)d}} \right)^{N/2} \, dx \right)^{2/N}.
\end{aligned}
\end{equation}
Choosing \( K \) sufficiently large so that the terms in \eqref{eq:I_1_estimate_moser} and \eqref{eq:I_2_estimate_moser} are absorbed into the right-hand side of \eqref{eq:moser_estimate1}, we deduce the following estimate
\begin{equation}
\left( \int_{\Omega} \frac{|w|^{2^*} |w_k|^{\beta \frac{2^*}{2}}}{|x|^{2^* d}} \, dx \right)^{\frac{2}{2^*}}
\leq 
C \left[
K^\beta \int_{\Omega} \frac{|w|^{2^*}}{|x|^{2^* d}} \, dx 
+ K^\beta \int_{\om} \frac{|w|^{p + 1}}{|x|^{(p+1)d}} \, dx 
\right].
\label{eq:estimate_absorbed}
\end{equation}
Now, applying Fatou's Lemma to the inequality \eqref{eq:estimate_absorbed}, we conclude the following for any \( \beta \geq 1 \)

\begin{equation*}
\int_{\Omega} \frac{|w|^{2^*(1 + \beta/2)}}{|x|^{2^* d}} \, dx < \infty. \qedhere
\end{equation*}
\end{proof}

\begin{theorem}[Uniform Boundedness] \label{thm:uniform_boundedness}
If \( w \) is a positive solution of \eqref{eq:transformed_equation}, then $w \in L^\infty(\om)$.
\end{theorem}
\begin{proof}
The proof follows directly from Lemmas \ref{lem:moser_iter_lem1} and \ref{lem:moser_iter_lem2}.
\end{proof}

\begin{proof}[Proof of Theorem \ref{thm:asympototic_estimates}]
 The proof follows directly from Theorems \ref{thm:harnack_inequality_transformed} and \ref{thm:uniform_boundedness}.   
\end{proof}

Finally, we have the following strong maximum principle. 
\begin{theorem}\label{thm:uniform_estimate_and_smp}
Assume that there exists a nonnegative solution $w_0 \in \X(\Omega)$ to the problem \eqref{eq:main_problem_ep} for some $\lambda \geq 0$ and $p \in [1,2^*-1)$. Then $w_0 >0$ almost everywhere in $\Omega$.
\end{theorem}
\begin{proof}
Consider the inner product
\begin{equation*}
\langle w_0, \varphi \rangle + \varepsilon\langle w_0, \varphi \rangle_{s} = \int_{\Omega} \frac{w_0}{|x|^2} \varphi + \int_{\Omega} |w_0|^{2^*-2} w_0 \varphi + \lambda \int_{\Omega} |w_0|^{p-2} w_0 \varphi \geq 0.
\end{equation*}
Since $w_0 \geq 0$, applying Theorem 8.4 of \cite{garain_higher_holder_regularity}, we conclude that $w_0 > 0$.
\end{proof}


\section{Linear and Superlinear Case}\label{sec:linear_sublinear}
\subsection{Linear Case (\texorpdfstring{$p = 2$})}\label{sec:linear}
In this subsection, we consider the case of linear perturbation in problem \eqref{eq:main_problem}. The analysis relies heavily on the spectral properties of the underlying operator, particularly the behaviour of its first eigenvalue. These properties play a crucial role in establishing the existence of solutions.



Let $V(\om) = \left\{ u \in \X(\Omega) \mid \int_{\Omega} |u|^{2^*} dx= 1 \right\}$. Then we take the restriction of the functional $\Je$ on $V(\om)$ defined as 
\begin{equation*}
    \Q(u) := \int_{\Omega} |\nabla u|^2dx - \mu \int_{\Omega} \frac{|u|^2}{|x|^2} dx + [u]_s^2 - \lambda \int_{\Omega} |u|^2dx.
\end{equation*}
Now define
\begin{equation*}
S_{\mu}(\lambda) := \inf_{u \in V(\om)} \Q(u).
\end{equation*}
Let us state and prove some of the important properties of $S_{\mu}(\lambda)$.
\begin{lemma}[Properties of $S_{\mu}(\lambda)$]\label{lem:best_constant_being_constant}
The function $S_{\mu}(\lambda)$ satisfies the following properties
\begin{enumerate}
    \item It satisfies the lower bound
    \begin{equation*}
    S_{\mu}(\lambda) \geq S_{\mu} - \lambda |\Omega|^{2/N}.
    \end{equation*}
    
    \item For every $\lambda > 0$, we have $S_{\mu}(\lambda) \leq S_{\mu}.$

    \item If $0 < t_2 < t_{1}$, then $ S_{\mu}(t_2) \geq S_{\mu}(t_1)$.
    \item $S_{\mu}(\la) \geq 0$ if and only if $0 < \la \leq \la_1$, where $\la_1$ is the first eigenvalue defined in \eqref{eq:first_eigenvalue}.
    \item For $0 < \lambda \leq \lambda_{1, s}$, we have $S_{\mu}(\lambda) = S_{\mu} > 0$, where $\la_{1,s}$ is the first eigenvalue defined in \eqref{eq:first_eigenvalue_frac}.
    \item The function $\lambda \mapsto S_{\mu}(\lambda)$ is continuous on $(0, \infty)$.
\end{enumerate}
\end{lemma}
\begin{proof}
Part $1 - 4$ follows directly. \\
Part $5$:  For $u \in  V(\om)$, we estimate
\begin{equation*}
\Q(u) \geq \int_{\Omega} |\nabla u|^2 - \mu \int_{\Omega} \frac{|u|^2}{|x|^2} dx + \lambda_{1,s} \int_{\Omega} |u|^2 - \lambda \int_{\Omega} |u|^2 \geq \int_{\Omega} |\nabla u|^2 - \mu \int_{\Omega} \frac{|u|^2}{|x|^2}d x.
\end{equation*}
This implies $S_{\mu}(\lambda) \geq  S_{\mu}.$ Using this with part $2$, we conclude $S_{\mu}(\lambda) = S_{\mu}$ for $0 < \lambda \leq \lambda_{1, s}$.\\
Finally part $6$ follows similarly to Lemma $4.8$ in \cite{biagi2022brezis}.
\end{proof}

\begin{proof}[Proof of Theorem \ref{thm:existence_linear_case}]
Define
\begin{equation*}
\lambda^* = \sup \left\{ \lambda_0 > 0 \mid S_{\mu}(\lambda) = S_{\mu} \mbox{ for all } 0 < \la \leq \la_0 \right\}.
\end{equation*}
By part $5$ of Lemma \ref{lem:best_constant_being_constant}, it is evident that $\lambda^* \geq \lambda_{1, s}$. Furthermore, since $S_{\mu}(\lambda)$ is a continuous function, we obtain $S_{\mu}(\lambda^*) = S_{\mu}$. Moreover, from the properties of $S_{\mu}(\lambda)$, we deduce
\begin{equation*}
S_{\mu}(\lambda) 
\begin{cases}
    \geq 0, & \text{if } \lambda \leq \lambda_1, \\
    \leq 0, & \text{if } \lambda > \lambda_1.
\end{cases}
\end{equation*}
This implies $S_{\mu}(\lambda_1) = 0$. Thus, we conclude $\lambda^* \in [\lambda_{1, s}, \lambda_1)$.

\noi \textbf{Case 1:} $0 < \lambda \leq \lambda_{1, s}$ \\
In this range, there does not exist a solution to the problem $( \mathcal{P}_{\mu,\la})$ within the closed ball $\mathcal{B}$ defined in \eqref{eq:linear_ball_NE}. On the contrary, $u$ be the solution to $( \mathcal{P}_{\mu,\la})$. Let $ v = \frac{u}{\| u \|_{L^{2^*}(\Omega)}}$. Then
\begin{equation*}
\Q(v)  = \frac{1}{\| u \|_{L^{2^*}}^2} \left( \int_{\om} |\grad u|^2 dx + [u]_s^2 - \mu \int_{\Omega} \frac{u^2}{|x|^2} dx - \lambda \| u \|_{L^2}^2 \right) = \frac{1}{\| u \|_{L^{2^*}}^2} \| u \|_{L^{2^*}}^{2^*}\leq S_{\mu}.
\end{equation*}
However, by part $5$ of Lemma \ref{lem:best_constant_being_constant} yield $\Q(u) \geq S_{\mu}.$ Thus, we conclude that $\Q(v) = S_{\mu}$. Further, we analyze
\begin{equation*}
S_{\mu}  \leq \| \nabla v \|_{L^2}^2 - \mu \int_{\Omega} \frac{v^2}{|x|^2} = \Q(v) - \left( [v]^2_{s} - \lambda \| v \|_{L^2}^{2} \right)\leq \Q(v)-\left( \lambda_{1,s} - \lambda \right) \| v \|_{L^2}^{2} \leq \Q(v) \leq S_{\mu}.
\end{equation*}
But $S_{\mu}$ is never achieved in $\Omega \subsetneq \mathbb{R}^N$ (bounded domain), we arrive at a contradiction.

\noi \textbf{Case 2:} $\lambda^* < \lambda < \lambda_1$ \\
For this range, by definition of $\la^*$, we have $0 \leq S_{\mu}(\lambda) < S_{\mu}$. By Lemma 1.2 of \cite{brezis_1983}, $S_{\mu}(\la)$ is achieved. Thus, there exists a nonzero function $w \in V(\Omega)$ such that $\Q(w) = S_{\mu}(\lambda)$. Since $\lambda < \lambda_1$, we establish $S_{\mu}(\lambda)  \geq (\lambda_1 - \lambda) \| w \|_{L^2}^2 > 0$. Since, $\Q(|w|) \leq  \Q(w)$. Without loss of generality, we assume $w \geq 0$ almost everywhere in $\Omega$.
Applying the method of Lagrange multipliers, there exists a scalar $\theta$ such that
\begin{equation*}
\nabla w  \nabla \varphi + \langle w, \varphi \rangle_{s} - \mu \int_{\Omega} \frac{w \varphi}{|x|^2}  - \lambda \int_{\Omega} w \varphi = \theta \int_{\Omega} |w|^{2^*-2}w \varphi, \quad \forall \varphi \in \X(\Omega).
\end{equation*}
Choosing $\varphi = w$ gives $\theta = S_{\mu}(\lambda)$.
Now, letting $u = S_{\mu}(\lambda)^{(N-2)/4}w$, we obtain
\begin{equation*}
\int_{\Omega} \nabla u \nabla \varphi + \langle u, \varphi \rangle_{s} - \mu \int_{\Omega} \frac{u \varphi}{|x|^2} \varphi = \int_{\Omega} |u|^{2^*-2} u \varphi + \lambda \int_{\Omega} u \varphi.
\end{equation*}
Thus, $u$ is a solution to the problem $( \mathcal{P}_{\mu,\la})$.


\noi \textbf{Case 3:} $\lambda \geq \lambda_1$ where $\la_1$ is the eigenvalue of \eqref{eq:first_eigenvalue} corresponding to the eigenfunction $\phi_0$\\
Since $S_{\mu}(\lambda)$ is a non-increasing function, we obtain $S_{\mu}(\lambda) \leq 0 < S_{\mu}$. By Lemma 1.2 of \cite{brezis_1983}, $S_{\mu}(\lambda)$ is achieved.
Now, on contrary, let $u$ be a positive solution of $(\mathcal{P}_{\mu,\la})$. 
Then using the strong maximum principle in Theorem \ref{thm:uniform_estimate_and_smp}, we write
\begin{equation*}
0 < \int_{\Omega} (u^{2^*-1} + \lambda u) \varphi_0 = \int_{\Omega} \nabla \varphi_0   \nabla u + \langle \varphi_0, u \rangle_s - \mu \int_{\Omega} \frac{\varphi_0 u}{|x|^2} =  \lambda_1 \int_{\Omega} \varphi_0 u.
\end{equation*}
So, we conclude that $\lambda_1 > \lambda$, which is a contradiction.
This completes the proof of Theorem \eqref{thm:existence_linear_case}.
\end{proof}

\subsection{Superlinear Case (\texorpdfstring{$2 < p < 2^*$})}\label{sec:superlinear}

Let \( 2 < p < 2^* \). Since we are concerned with positive solutions of \eqref{eq:main_problem}, it is not necessary that \(|u|\) is also a solution whenever \( u \) is a solution. To address this issue, we consider the following modified functional
\begin{equation*}
\Jplus(u) = \frac{1}{2} \int_{\om} |\grad u|^2 + \frac{1}{2} [u]_s^2 - \frac{\mu}{2} \int_{\Omega} \frac{|u|^2}{|x|^2} dx 
- \frac{1}{2^*} \int_{\Omega} (u^+)^{2^*} - \frac{\lambda}{p} \int_{\Omega} (u^+)^{p}.
\end{equation*}
\noi Any critical point of the functional \( \Jplus \) is a solution of the following equation
\begin{equation}\tag{$\mathcal{P}_{\mu, \la}^{+}$}
\begin{cases}
-\Delta u + (-\Delta)^s u - \mu \frac{u}{|x|^2} = (u^+)^{2^*-1} + \lambda (u^+)^{p-1}, & \text{in } \Omega, \\
u = 0, & \text{in } \mathbb{R}^N \setminus \Omega.
\end{cases}
\label{eq:critical_equation}
\end{equation}
To link the solutions of equation \eqref{eq:critical_equation} with the solutions of problem \eqref{eq:main_problem}, we use the following Weak Maximum Principle.
\begin{theorem}[Weak Maximum Principle]\label{thm:weak_maximum_principle}
Let \( u \in \X(\Omega) \) be a supersolution of \eqref{eq:critical_equation}. Then \( u \geq 0 \) almost everywhere in \( \Omega \).
\end{theorem}
\begin{proof}
On contrary assume that $u \not \geq 0 $ everywhere in $\om$. Then there exists a set \( E \) of positive measure such that \( u < 0 \) in \( E \). 
Take \( u^- \) as a test function and Hardy inequality, we have
\begin{equation}\label{eq:nonnegative_ip1}
\begin{aligned}
0  &\leq \int_{\Omega} \nabla u  \nabla u^- + \langle u, u^- \rangle_{s} - \mu \int_{\Omega} \frac{u u^-}{|x|^2}  = -\int_{\Omega} |\nabla u^-|^2 + \mu \int_{\Omega} \frac{(u^-)^2}{|x|^2} + \langle u, u^-\rangle_s \\
& \leq - \left(1 - \frac{\mu}{\bar{\mu}} \right) \int_{\Omega} |\nabla u^-|^2 +  \langle u, u^-\rangle_s \leq \langle u , u^- \rangle_s.
\end{aligned}
\end{equation}
Moreover, using inner product properties, we obtain
\begin{equation*}
\langle u^-, u^- \rangle_{s} = \iint_{\mathbb{R}^{2N}} \frac{|u^-(x) - u^-(y)|^2}{|x - y|^{N+2s}} dy\, dx
\geq \iint_{E \times (\mathbb{R}^N \setminus \Omega)} \frac{|u^-(x)|^2}{|x - y|^{N+2s}} dy\, dx > 0.
\end{equation*}
Thus, we deduce
\begin{equation*}
\langle u, u^- \rangle_{s} = \langle u^+, u^- \rangle_{s} - \langle u^-, u^- \rangle_{s}
< \langle u^+, u^- \rangle_{s} \leq 0.
\end{equation*}
Combining this with \eqref{eq:nonnegative_ip1}, we arrive at a contradiction. 
\end{proof}

Applying the Weak Maximum Principle in Theorem \ref{thm:weak_maximum_principle}, we obtain \( u_0 \geq 0 \) almost everywhere in \( \mathbb{R}^N \), with $u_0$ being solution of equation \eqref{eq:critical_equation} holds. This implies $(u_0)_+ \equiv u_0$.
Thus, \( u_0 \) is a solution to problem \eqref{eq:main_problem}.
Hence, it is now sufficient to establish the existence of a nonzero critical point of the functional \( \Jplus \). To this end, we first prove that the functional $\J^+$ satisfies the mountain pass geometry.


\begin{lemma}\label{lem:mountain_pass_superlinear}
The functional $\J^+$ satisfies the following
\begin{enumerate}
    \item[$(i)$] There exists $\al, \rho > 0$ such that for any $u \in \X(\om)$ with $\lv u \rv_{\X} = \rho$ we have $\J^+(u) \geq \al$;
    \item[$(ii)$] for any $u \in \X(\om)$ we have $\J^+(tu) \to - \infty$ as $ t \to \infty$.
\end{enumerate}
\end{lemma}
\begin{proof}
Sobolev inequality and Hardy inequality yield
\begin{equation*}
    \J^+(u) \geq \frac{1}{2} \left( 1 - \frac{\mu}{\bar \mu} \right) \lv u \rv_{\X}^2 - C \lv u \rv_{\X}^{2^*} - C \la \lv u \rv_{\X}^p.
\end{equation*}
Then $(i)$ follows directly. On the other hand
\begin{equation*}
    \J^+(tu) \leq \frac{t^2}{2} \lv u \rv_{\X}^2 - \frac{t^{2^*}}{2^*} \int_{\om} (u^+)^{2^*} - \la \frac{ t^p}{p} \int_{\om} (u^+)^p \to - \infty \quad \text{as } t \to \infty.
\end{equation*}
This concludes the proof of $(ii)$.
\end{proof}

\begin{lemma}\label{lem:ps_holds_superlinear}
The functional \( \Jplus \) satisfies the (PS)\(_c\) condition for every \( c < \frac{1}{N} S_{\mu}^{N/2} \). 
\end{lemma}
\begin{proof}
We begin by evaluating
\begin{equation*}
    \Jplus(u_n) - \frac{1}{2^*} \langle (\Jplus)^{\prime}(u_n), u_n \rangle = \frac{1}{N} \left( \int_{\Omega} |\nabla u_n|^2 \, dx + [u_n]^2_{s}  - \mu \int_{\Omega} \frac{|u_n|^2}{|x|^2} \, dx \right) 
    - \lambda \left( \frac{1}{p} - \frac{1}{2^*} \right) \|u_n\|_{L^p}^p.
\end{equation*}
Since \( \{u_n\} \) is a (PS)\(_c\) sequence, we obtain
\begin{equation*}
    C \left(1 + \|u_n\|_{{\X}}\right) \geq \frac{1}{N} \left(1 - \frac{\mu}{\bar{\mu}} \right) \|u_n\|_{\X}^2 
    - \lambda \left( \frac{1}{p} - \frac{1}{2^*} \right) \|u_n\|_{\X}^p,
\end{equation*}
which implies that the sequence \( \{u_n\} \) is bounded. Thus, there exists \( u \in \X(\Omega) \) such that $u_n \rightharpoonup u$ weakly in $\X(\om)$, $ u_n \rightarrow u$ strongly in $L^r(\om)$ for $r \in [1, 2^*)$, and $u_n(x) \rightarrow u(x)$ pointwise a.e. in $\om$.
\noi Now,
\begin{align*}
    &o_n(1)  = \langle (\Jplus)^{\prime}(u_n), u_n - u \rangle \\
    & = \left( \|\nabla u_n\|^2_2 - \|\nabla u\|^2_2 \right) +  \left( [u_n]^2_{s} - [u]^2_{s} \right)
    - \mu \int_{\Omega} \frac{|u_n|^2 - |u|^2}{|x|^2} \, dx - (\|u_n^+\|_{L^{2^*}}^{2^*} - \|u^+\|_{L^{2^*}}^{2^*}) + o_n(1) \\
    & = \|\nabla u_n - \nabla u\|^2_2 +  [u_n - u]^2_{s} 
    - \mu \int_{\Omega} \frac{|u_n - u|^2}{|x|^2} \, dx - \|u_n^+ - u^+\|_{L^{2^*}}^{2^*} + o_n(1),
\end{align*}
where the final equality follows from the Brezis–Lieb Lemma. Thus, we conclude that
\begin{equation*}
    \lim_{n \to \infty} \left( \|u_n - u\|^2_{\X} - \mu \int_{\Omega} \frac{|u_n - u|^2}{|x|^2} \, dx \right)
    = \lim_{n \to \infty} \|u_n^+ - u^+\|_{L^{2^*}}^{2^*} = \ell  (\text{say}).
\end{equation*}
By \eqref{eq:spectral_ratio_frac}, we also have
\begin{equation*}
   S_{\mu} \|u_n^+ - u^+\|_{L^{2^*}}^2 \leq S_{\mu} \|u_n - u\|_{L^{2^*}}^2 \leq \|u_n - u\|^2_{\X} - \mu \int_{\Omega} \frac{|u_n - u|^2}{|x|^2} \, dx.
\end{equation*}
This implies $ S_\mu \, \ell^{\frac{2}{2^*}} \leq \ell$.  If \( \ell > 0 \), then $\ell \geq S_\mu^{\frac{N}{2}}$. Next, consider
\begin{align*}
    \Jplus(u_n) - \frac{1}{2} \langle (\Jplus)^{\prime}(u_n), u_n \rangle 
    &= \left( \frac{1}{2} - \frac{1}{2^*} \right) \|u_n^+\|_{L^{2^*}}^{2^*} + \lambda \left( \frac{1}{2} - \frac{1}{p} \right)\int_{\Omega} |u_n^+|^p  dx \geq \frac{1}{N}\|u_n^+\|_{L^{2^*}}^{2^*}.
\end{align*}
This gives $c \geq \frac{1}{N}S_\mu^{\frac{N}{2}}$.
\end{proof}

Let us define the following parameters
\begin{equation*}
    \bar{\mu} = \left( \frac{N-2}{2} \right)^2, \quad \gamma = \sqrt{\bar{\mu}} + \sqrt{\bar{\mu} - \mu}, \quad \gamma^{\prime} = \sqrt{\bar{\mu}} - \sqrt{\bar{\mu} - \mu}.
\end{equation*}
For $\varepsilon > 0$, we define the function $U_{\varepsilon,\al}(x)$ as follows
\begin{equation*}
    U_{\varepsilon,\al}(x) =  \frac{B_0 \varepsilon^{ \alpha \sqrt{\bar{\mu}}}}{\left( \varepsilon^{2\alpha} |x|^{\gamma^{\prime}/\sqrt{\bar{\mu}}} + |x|^{\gamma/\sqrt{\bar{\mu}}} \right)^{\sqrt{\bar{\mu}}}},
\end{equation*}
where the constant $B_0 = \left( \frac{4N(\bar{\mu} - \mu)}{N-2} \right)^{(N-2)/4}$ and the parameter $\al > 0$ is chosen later.
The function $U_{\varepsilon,\al}\in D^{1,2}(\mathbb{R}^N)$ satisfies the following partial differential equation
\begin{equation*}
    -\Delta U - \frac{\mu}{|x|^2} U = |U|^{2^*-2} U.
\end{equation*}
Furthermore, $U_{\varepsilon,\al}$ achieves the Sobolev constant $S_{\mu}$, i.e.,
\begin{equation*}
    \int_{\mathbb{R}^N} \left( |\nabla U_{\varepsilon,\al}|^2 - \frac{\mu}{|x|^2} U_{\varepsilon,\al}^2 \right) dx = \int_{\mathbb{R}^N} |U_{\varepsilon,\al}|^{2^*} dx = S_{\mu}^{N/2}.
\end{equation*}
Next define $u_{\varepsilon,\al}(x) = \phi (x) U_{\varepsilon,\al}(x)$ where $\phi \in C_{c}^2(\om)$ is a cut-off function satisfying $0 \leq \phi(x) \leq 1$, $\phi \equiv 1$ in $B_{\delta} $, $Supp(\phi) \subset B_{2\delta} \subset \om$, and $|\grad \phi |\leq C$. Then from \cite{chen2003estimates}, we have the following estimates 
\begin{align}
\int_{\Omega} |\nabla u_{\varepsilon,\al}|^2 - \mu \int_{\Omega} \frac{u_{\varepsilon,\al}^2}{|x|^2} dx&= S_{\mu}^{N/2} + O(\varepsilon^{\al (N-2)}),
\label{eq:norm_upper_bound}\\
\| u_{\varepsilon,\al} \|^{2^*}_{L^{2^*}} &\geq S_{\mu}^{N/2} - O( \varepsilon^{\al N} ),
\label{eq:norm_lower_bound}\\
\int_{\om}|u_{\varepsilon,\al}|^{p}dx &= O(\varepsilon^{ \al (N - \sqrt{\bar\mu}p)\frac{\sqrt{\bar\mu}}{\sqrt{\bar\mu - \mu}}}),\label{eq:lower_bound_norm_r}
\end{align}
provided that $\mu < \bar{\mu} -1$ and $p > \frac{N}{\gamma}$. \\
We now estimate the Gagliardo norm of $u_{\varepsilon,\al}(x)$. Using the estimate $(6.12)$ along with Exercise $1.26$ of \cite{giovanni_fractional_book} and \eqref{eq:norm_upper_bound} we deduce
\begin{equation}\label{eq:fractional_norm_bound}
\begin{aligned}
\iint_{\mathbb{R}^{2N}} \frac{| u_{\varepsilon,\al}(x) - u_{\varepsilon,\al}(y)|^2}{|x - y|^{N+2s}} dx \, dy & \leq C \left( \int_{\rnn} |u_{\varepsilon,\al}|^2 \, dx \right)^{1-s}\left( \int_{\rnn} |\grad u_{\varepsilon,\al}|^2\,dx\right)^s\\
& \leq C \varepsilon^{2\al(1-s) \frac{\sqrt{\bar{\mu}}}{\sqrt{\bar \mu - \mu}}} \left( \int_{\rnn} |\grad u_{\varepsilon,\al}|^2\,dx - \mu \int_{\Omega} \frac{u_{\varepsilon,\al}^2}{|x|^2} + \mu \int_{\rnn} \frac{U_{\varepsilon,\al}^2}{|x|^2} \right)^s\\
& \leq C \varepsilon^{2\al(1-s) \frac{\sqrt{\bar{\mu}}}{\sqrt{\bar \mu - \mu}}} \left(  S_{\mu}^{N/2} + O(\varepsilon^{\al (N-2)}) + \mu \int_{\rnn} \frac{U_{1,\al}^2}{|x|^2} \right)^s\\
& \leq C \varepsilon^{2\al(1-s) \frac{\sqrt{\bar{\mu}}}{\sqrt{\bar \mu - \mu}}}
\end{aligned}
\end{equation}

To conclude, we use the mountain pass theorem with the help of functions $u_{\varepsilon,\al}$ defined above. Fixing $\al =1$ and defining $u_{\varepsilon} = u_{\varepsilon, \alpha}$, we combine the estimates in \eqref{eq:norm_upper_bound}, \eqref{eq:norm_lower_bound}, \eqref{eq:lower_bound_norm_r} with $r = p$ and \eqref{eq:fractional_norm_bound} to obtain
\begin{equation*}
\begin{aligned}
\Jplus(t u_{\varepsilon}) &  \leq \frac{t^2}{2} \left[ S_{\mu}^{N/2} + C \varepsilon^{(N-2)}  \right] - \frac{t^{2^*}}{2^*}\left[ S_{\mu}^{N/2} - C \varepsilon^{ N}  \right] -\frac{\la t^{p}}{p} C \varepsilon^{\frac{ \sqrt{\bar{\mu}} {(N - p\sqrt{\bar{\mu}})}}{{\sqrt{\bar{\mu} - \mu}}}} + \frac{t^2}{2} C \varepsilon ^{ \frac{2  (1-s)\sqrt{\bar{\mu}}}{\sqrt{\bar{\mu} - \mu}}}\\
& \leq \frac{t^2}{2} \left[ S_{\mu}^{N/2} + C \varepsilon^{\beta_{\mu,N,s}}  \right] - \frac{t^{2^*}}{2^*}\left[ S_{\mu}^{N/2} - C \varepsilon^{ N}  \right] -\frac{\la t^{p}}{p} C \varepsilon^{\frac{ \sqrt{\bar{\mu}} {(N - p\sqrt{\bar{\mu}})}}{{\sqrt{\bar{\mu} - \mu}}}} := g(t)
\end{aligned}
\end{equation*}
where $\beta_{\mu,N,s}$ is defined in \eqref{eq:superlinear_parameter}.
Observing the behavior of \( g(t) \)
\begin{equation*}
g(0) = 0, \quad \text{and} \quad g(t) \to -\infty \text{ as } t \to \infty.
\end{equation*}
Thus, there exists \( t_{\varepsilon, \lambda} > 0 \) such that
\begin{equation*}
\sup_{t > 0} g(t) = g(t_{\varepsilon, \lambda}).
\end{equation*}
If \( t_{\varepsilon, \lambda} = 0 \), then there is nothing to prove. Otherwise, we take \( t_{\varepsilon, \lambda} > 0 \).
Differentiating \( g(t) \) and equating it to zero
\begin{equation*}
0 = g'(t_{\varepsilon, \lambda}) = t_{\varepsilon, \lambda} \left[ S_{\mu}^{N/2} + C \varepsilon^{\beta_{\mu,N,s}} \right] - t_{\varepsilon, \lambda}^{2^*-1} \left[ S_{\mu}^{N/2} - C \varepsilon^{ N} \right] - \lambda t_{\varepsilon, \lambda}^{p-1} C \varepsilon^{\frac{ \sqrt{\bar{\mu}} {(N - p\sqrt{\bar{\mu}})}}{{\sqrt{\bar{\mu} - \mu}}}}.
\end{equation*}
Rearranging terms
\begin{equation}\label{eq:g_derivative_0}
S_{\mu}^{N/2} + C \varepsilon^{ \beta_{\mu,N,s}} = t_{\varepsilon, \lambda}^{2^*-2} \left[ S_{\mu}^{N/2} - C \varepsilon^{ N} \right] + \lambda t_{\varepsilon, \lambda}^{p-2} C \varepsilon^{\frac{ \sqrt{\bar{\mu}} {(N - p\sqrt{\bar{\mu}})}}{{\sqrt{\bar{\mu} - \mu}}}}.
\end{equation}
This gives $\displaystyle\liminf_{\varepsilon \to 0} t_{\varepsilon,\la} > 0$. Indeed if $\displaystyle\liminf_{\varepsilon \to 0} t_{\varepsilon,\la} =0$, we arrive at $S_{\mu}^{N/2} = 0$ which is absurd. In addition, from \eqref{eq:g_derivative_0}, we derive an upper bound for \( t_{\varepsilon, \lambda} \) as
\begin{equation*}
t_{\varepsilon, \lambda} < \left( \frac{S_{\mu}^{N/2} + C \varepsilon^{\beta_{\mu,N,s}}}{S_{\mu}^{N/2} - C \varepsilon^{ N}} \right)^{\frac{1}{2^*-2}} = 1 + C \varepsilon^{\beta_{\mu,N,s}}= h(\varepsilon).
\end{equation*}
As \( \varepsilon \to 0 \), it follows that \( h(\varepsilon) \to 1 \) and $t_{\varepsilon, \lambda} \geq \theta_{\la} > 0$ for some $\theta_\la$. Furthermore, consider the function
\begin{equation*}
h_1(t) = \frac{t^2}{2} \left[ S_{\mu}^{N/2} + C \varepsilon^{\beta_{\mu,N,s}} \right] - \frac{t^{2^*}}{2^*} \left[ S_{\mu}^{N/2} - C \varepsilon^{ N} \right].
\end{equation*}
Then $h_1(t)$ is increasing on \( (0, h(\varepsilon)] \). Thus, we obtain
\begin{align*}
\sup_{t>0} g(t) & = g(t_{\varepsilon,\la})  \\
& \leq   S_{\mu}^{N/2} \left[ \frac{h(\varepsilon)^2}{2} - \frac{h(\varepsilon)^{2^*}}{2^*} \right] + \frac{C}{2} h(\varepsilon)^2 \varepsilon^{\beta_{\mu,N,s}} + \frac{h(\varepsilon)^{2^*} C \varepsilon^{ N}}{2^*}  -\frac{\la t_{\varepsilon,\la}^{p}}{p} C \varepsilon^{\frac{ \sqrt{\bar{\mu}} {(N - p\sqrt{\bar{\mu}})}}{{\sqrt{\bar{\mu} - \mu}}}}\\
& \leq   S_{\mu}^{N/2} \left(1 + \varepsilon^{\beta_{\mu,N,s}} \right)^{\frac{2}{2^*-2}} \left[ \frac{1}{2} - \frac{1 + \varepsilon^{\beta_{\mu,N,s}}}{2^*} \right] + C \varepsilon^{\beta_{\mu,N,s}}  -\la C \theta_{\la}^p \varepsilon^{\frac{ \sqrt{\bar{\mu}} {(N - p\sqrt{\bar{\mu}})}}{{\sqrt{\bar{\mu} - \mu}}}}\\
& = \frac{1}{N} S_{\mu}^{N/2} +  C \varepsilon^{\beta_{\mu,N,s}}  -\la \theta_{\la}^p C \varepsilon^{\frac{ \sqrt{\bar{\mu}} {(N - p\sqrt{\bar{\mu}})}}{{\sqrt{\bar{\mu} - \mu}}}}.
\end{align*}
\textbf{Case 1:} For the case $\beta_{\mu,N,s} > \frac{ \sqrt{\bar{\mu}} {(N - p\sqrt{\bar{\mu}})}}{{\sqrt{\bar{\mu} - \mu}}}$, it follows that for sufficiently small \( \varepsilon > 0 \), we obtain
\begin{equation*}
\sup_{t \geq 0} \Jplus(t u_{\varepsilon}) < \frac{1}{N} S_{\mu}^{N/2}.
\end{equation*}
\textbf{Case 2:} When $\beta_{\mu,N,s} \leq \frac{ \sqrt{\bar{\mu}} {(N - p\sqrt{\bar{\mu}})}}{{\sqrt{\bar{\mu} - \mu}}}$, we claim that $\displaystyle \lim_{\lambda \to \infty} t_{\varepsilon, \lambda} = 0$. On contrary
\begin{equation*}
l = \limsup_{\lambda \to \infty} t_{\varepsilon, \lambda} > 0.
\end{equation*}
Choose a sequence \( \lambda_k \to \infty \) as \( k \to \infty \) such that $t_{\varepsilon, \lambda_k} \to l$. Since \( \lambda_k \to \infty \), the right-hand side (RHS) of equation \eqref{eq:g_derivative_0} approaches infinity, which contradicts the left-hand side being finite. This contradiction establishes the claim.

\begin{proof}[Proof of Theorem \ref{thm:superlinear_pb}]
With the above discussion and Lemma \ref{lem:mountain_pass_superlinear}, the hypotheses of the mountain pass theorem hold. This gives the existence of a $(PS)_c$ sequence and Lemma \ref{lem:ps_holds_superlinear} gives that $(PS)_c$ condition is satisfied with $c < \frac{1}{N} S_{\mu}^{N/2}$. This completes the proof. 
\end{proof}

\section{Sublinear Case (\texorpdfstring{$1<p<2$})}\label{sec:sublinear}
This section is devoted to the sublinear case, that is, \( 1 < p < 2 \) in problem~\eqref{eq:main_problem_ep}. The presence of both convex and concave terms enriches the variational structure of the associated energy functional, allowing us to exploit its topology to establish the existence of two distinct positive solutions. Since we are concerned with positive solutions of \eqref{eq:main_problem_ep}, we follow in the same way as superlinear case and consider the following modified functional
\begin{equation*}
\Jep(u) = \frac{1}{2} \int_{\om} |\grad u|^2 + \frac{\varepsilon}{2} [u]_s^2 - \frac{\mu}{2} \int_{\Omega} \frac{|u|^2}{|x|^2} dx 
- \frac{1}{2^*} \int_{\Omega} (u^+)^{2^*} - \frac{\lambda}{p} \int_{\Omega} (u^+)^{p}.
\end{equation*}
\noi Any critical point of the functional \( \Jep \) is a solution of the following equation

\begin{equation}\tag{$\mathcal{P}_{\mu, \la}^{\varepsilon,+}$}
\begin{cases}
-\Delta u + \varepsilon(-\Delta)^s u - \mu \frac{u}{|x|^2} = (u^+)^{2^*-1} + \lambda (u^+)^{p-1}, & \text{in } \Omega, \\
u = 0, & \text{in } \mathbb{R}^N \setminus \Omega.
\end{cases}
\label{eq:maineq_plus_epsilon}
\end{equation}

Applying the Weak Maximum Principle, we obtain \( u_0 \geq 0 \) almost everywhere in \( \mathbb{R}^N \), with $u_0$ being solution of equation \eqref{eq:maineq_plus_epsilon} holds. This implies $(u_0)_+ \equiv u_0$. Thus, \( u_0 \) is a solution to problem \eqref{eq:main_problem_ep}.
Hence, it is now sufficient to establish the existence of a nonzero critical point of the functional \( \Jep \).

We begin by proving the existence of the first positive solution to problem~\eqref{eq:main_problem_ep}. First we prove a lemma that gives the convergence of (PS) sequence. 
\begin{lemma}\label{lem:ps_sublinear}
If \( p \in (1, 2) \), then the (PS)\(_c\) condition for the functional \( \Jep \) holds for all    
\begin{equation*}
    c < \frac{1}{N} S^{N/2} 
    - |\Omega| \left(1 - \frac{ p}{2^*} \right) \left( \frac{pN}{2^*} \right)^{-p/(2^* - p)}
    \left[ \lambda \left( \frac{1}{p} - \frac{1}{2} \right) \right]^{2^*/(2^* - p)}.
\end{equation*}
\end{lemma}
\begin{proof}
Following in the same way as in Lemma \ref{lem:ps_holds_superlinear}, we obtain
\begin{align*}
    \Jep(u_n) - \frac{1}{2} \langle (\Jep)^{\prime}(u_n), u_n \rangle 
    &= \left( \frac{1}{2} - \frac{1}{2^*} \right) \|u_n^+\|_{L^{2^*}}^{2^*} - \lambda \left( \frac{1}{p} - \frac{1}{2} \right)\int_{\Omega} |u_n^+|^p  dx.
\end{align*}
By the definition of (PS)\(_c\) sequence, H\"{o}lder's inequality, and Young’s inequality, we obtain
\begin{align*}
    c + o_n(1) &\geq \left( \frac{1}{2} - \frac{1}{2^*} \right) \left[ \ell + \|u^+\|_{L^{2^*}}^{2^*} \right]
    - \lambda \left( \frac{1}{p} - \frac{1}{2} \right) \|u^+\|_{L^p}^p \\
    &\geq \left( \frac{1}{2} - \frac{1}{2^*} \right) \left[ \ell + \|u^+\|_{L^{2^*}}^{2^*} \right]
    - \lambda \left( \frac{1}{p} - \frac{1}{2} \right) |\om|^{(2^*-p)/2^*} \|u^+\|_{L^{2^*}}^p \\
    &\geq \frac{\ell}{N}
    - |\Omega| \left(1 - \frac{ p}{2^*} \right) \left( \frac{pN}{2^*} \right)^{-p/(2^* - p)}
    \left[ \lambda \left( \frac{1}{p} - \frac{1}{2} \right) \right]^{2^*/(2^* - p)}.
\end{align*}
This contradicts the definition of \( c \), so \( \ell = 0 \). Hence, we obtained the desired result.
\end{proof}
Now, let us define 
\begin{equation}\label{eq:Lambda}
    \La = \sup \{ \la > 0 : \eqref{eq:main_problem_ep} \text{ has a weak solution} \}.
\end{equation}
Then we have the following lemma.
\begin{lemma}
Let $\La$ be defined as in \eqref{eq:Lambda}. Then $0 < \La < \infty$.
\end{lemma}
\begin{proof}
Take $\varphi \in \X(\om)$ such that $\|\varphi\|_{\X} = 1$. Since $\Jep(t \varphi) \to -\infty$ as $t \to \infty$, there exists $t_0 > 0$ independent of $\varepsilon$ and $\la$ (as term corresponding to $\la$ can be dropped) such that $\Jep(t_0 \varphi) < 0$. Then, for $0 < \rho \leq t_0$, we must have
\begin{equation*}
c_{\lambda} = \inf_{u \in \overline{B_{\rho}}} \Jep(u) < 0.
\end{equation*}
Then we obtain a (PS) sequence from the minimizing sequence via the Ekeland's Variational Principle \cite[Theorem 8.5]{willem2012minimax}. Thus, by the Lemma \ref{lem:ps_sublinear}, $\Jep$ achieves its minimum $c_{\lambda}$ at some nonnegative function $u_{\lambda}$ for $\la$ sufficiently small. Thus $\La > 0.$\\
Next we prove that $\La  < \infty$. Consider the eigenvalue problem 
\begin{equation*}
\left\{
\begin{aligned}
- \Delta \phi_1 + \varepsilon  (-\Delta)^s \phi_1 - \mu \frac{\phi_1}{|x|^2} &= {\lambda_{1,\varepsilon}} \phi_1, && \text{in } \Omega, \\
\phi_1 &= 0, && \text{in } \mathbb{R}^N \setminus \Omega,
\end{aligned}
\right. \tag{$\mathcal{P}_1$}
\end{equation*}
Let $u$ to be nonnegative solution of the problem \eqref{eq:main_problem_ep}. Then we have
\begin{equation} \label{eq:equality_of_soln}
\int_{\om} u^{2^*-1} \phi_1 dx + \la \int_{\om} u^{p-1} \phi_1 dx = \int_{\om} \grad u  \grad \phi_1 + \varepsilon \langle u , \phi_1 \rangle_s - \mu \int_{\om} \frac{u \phi_1}{|x|^2} dx = \lambda_{1,\varepsilon} \int_{\om} u \phi_1 dx.
\end{equation}
We note that there exists $\la_*$ such that for all $\la > \la_*$ we have 
\begin{equation*}
    t^{2^* - 1} + \la t^{p-1} > \la_{1,\varepsilon} t 
\end{equation*}
for all $t > 0$. Combining this with \eqref{eq:equality_of_soln} we deduce that $\La < \infty$.
\end{proof}
\begin{lemma}\label{lem:existence_each_lambda}
    The problem \eqref{eq:main_problem_ep} has at least one positive solution for each $\la \in (0, \La)$. In fact, the sequence of minimal solutions $\{u_\la \}$ are increasing with respect to $\la$. If $\la = \La$, then the problem \eqref{eq:main_problem_ep} admits at least one weak solution.
\end{lemma}
\begin{proof}
Let \(\tilde\lambda \in (0, \La)\) be fixed. By definition of $\La$, there exists \(\lambda_{*} \in (\tilde\lambda, \Lambda)\) such that \eqref{eq:main_problem_ep} with $\la = \la_*$ has a positive weak solution \(\bar u\).  
Then \(\bar  u\) is a supersolution of the problem \eqref{eq:main_problem_ep} with $\la = \tilde \la$. To obtain the subsolution, we define the following eigenvalue problem
\begin{equation}
\begin{cases}
-\Delta \varphi_1 + \varepsilon(-\Delta)^{s} \varphi_1 = \lambda_{\varepsilon}^1 \varphi_1 & \text{in } \Omega, \\[6pt]
\varphi_1 = 0 & \text{in } \mathbb{R}^N \setminus \Omega,
\end{cases}
\label{eq:P2}
\tag{P2}
\end{equation}
where \(\lambda_{\varepsilon}^1\) is the first eigenvalue of the mixed operator \(-\Delta + \varepsilon(-\Delta)^s\) with eigenfunction $\varphi_1 \in L^{\infty}(\om)$. Then \(\varphi_t = t \varphi_1\) also solves \eqref{eq:P2}.  
Choosing \(t > 0\) sufficiently small such that
\begin{equation*}
-\Delta \varphi_t + \varepsilon (-\Delta)^s \varphi_t = \lambda_{\varepsilon}^1 \varphi_t 
\leq \tilde\lambda \varphi_t^{p-1} \leq \tilde\lambda \varphi_t^{p-1} + \varphi_t^{2^*-1} + \mu \frac{\varphi_t}{|x|^2}.
\end{equation*}
Hence, \(\varphi_t\) is a subsolution of \eqref{eq:main_problem_ep}. 
Next, taking \((\varphi_t - \bar u)^+\) as a test function, we obtain
\begin{equation*}
\begin{aligned}
0 &\leq \int_{\Omega} |\nabla (\varphi_t - \bar u)^+|^2 \, dx 
+  \varepsilon \big[ (\varphi_t - \bar u)^+ \big]^2_{s} \\
&\leq \int_{\Omega} \nabla (\varphi_t - \bar u)  \nabla (\varphi_t - \bar u)^+ \, dx 
+ \varepsilon \langle \varphi_t - \bar u, (\varphi_t - \bar u)^+ \rangle_{s} \\[6pt]
&= \lambda_{\varepsilon}^1 \int_{\Omega} \varphi_t (\varphi_t - \bar u)^+ \, dx
- \lambda_* \int_{\Omega} \bar u^{p-1} (\varphi_t - \bar u)^+ \, dx \\[6pt]
&\quad - \int_{\Omega} \bar u^{2^*-1} (\varphi_t - \bar u)^+ \, dx
- \mu \int_{\Omega} \bar u \, \frac{(\varphi_t - \bar u)^+}{|x|^2} \, dx\\
&\leq \lambda_* \int_{\Omega} \Big( \varphi_t^{p-1} - \bar u^{p-1} \Big) (\varphi_t - \bar u)^+ \, dx \leq 0,
\end{aligned}
\end{equation*}
which yields $\varphi_t \leq \bar u$. Now, let \(\{u_k\}\) denote the nonnegative sequence in \(\X(\Omega)\) of solutions to the iterated problem
\begin{equation*}
(P_k) \; \; \; 
\begin{cases}
-\Delta u_k + \varepsilon (-\Delta)^s u_k = \mu \dfrac{u_{k-1}}{|x|^2} + \tilde\lambda u_{k-1}^{p-1} + u_{k-1}^{2^*-1} & \text{in } \Omega, \\[8pt]
u_k = 0 & \text{in } \mathbb{R}^N \setminus \Omega
\end{cases}
\end{equation*}
for $k \geq 1$ and $u_0 = \varphi_t$. Then we claim that
\[
\varphi_t \leq u_1 \leq u_2 \leq \cdots \leq \overline{u} .
\]
We prove this by principle of mathematical induction. Using the same idea as above, it is clear that \(\varphi_t \leq u_1\). Assume that the claim is true for $k$ and then we prove it for $k+1$, i.e., $u_{k} \leq u_{k+1}$ whenever $u_{k-1} \leq u_k$ for all $k \in \ntrl$. Take \((u_{k} - u_{k+1})^+\) as a test function in \((P_k)\) and \((P_{k+1})\) to get 
\begin{equation*}
\begin{aligned}
0 &\leq \int_{\Omega} |\nabla (u_{k} - u_{k+1})^+|^2 \, dx 
+ \varepsilon \big[ (u_{k} - u_{k+1})^+ \big]^2_{s} \\[6pt]
&\leq \mu \int_{\Omega} \frac{(u_{k-1} - u_{k})(u_{k} - u_{k+1})^+}{|x|^2} \, dx + \tilde\lambda \int_{\Omega} \big(u_{k-1}^{p-1} - u_{k}^{p-1}\big) (u_{k} - u_{k+1})^+ \, dx \\[6pt]
&\quad + \int_{\Omega} \big(u_{k-1}^{2^*-1} - u_{k}^{2^*-1}\big)(u_{k} - u_{k+1})^+ \, dx \leq 0,
\end{aligned}
\end{equation*}
where the last inequality follows from induction hypothesis \(u_{k-1} \leq u_{k}\). Then it follows that $u_{k} \leq u_{k+1}$.
Since \(u_{k} \leq \overline{u}\), it follows in same way that $u_{k+1} \leq \overline{u}$. Hence, our claim is proved. \\
Next, define
\[
u_{\tilde\lambda} = \lim_{k \to \infty} u_k \quad \text{in } L^1(\Omega).
\]
Then 

\begin{equation*}
\begin{aligned}
\|u_k\|^2_{\Xe} 
&= \mu \int_{\Omega} \frac{u_k u_{k-1}}{|x|^2} \, dx 
+ \int_{\Omega} u_k u_{k-1}^{2^*-1} \, dx 
+ \tilde\lambda \int_{\Omega} u_k u_{k-1}^{p-1} \, dx \\[6pt]
&\leq \mu \int_{\Omega} \frac{\overline{u}^2}{|x|^2} \, dx
+ \int_{\Omega} \overline{u}^{2^*} \, dx 
+ \tilde\lambda \int_{\Omega} \overline{u}^p \, dx \leq C.
\end{aligned}
\end{equation*}
Thus, up to a subsequence, we know that \(u_k \rightharpoonup u_{\tilde\lambda}\) in \(\X(\Omega)\).  
Therefore, by passing to the limit in \((P_k)\), we conclude that \(u_{\tilde\lambda} \geq 0\) is the solution of the problem \eqref{eq:main_eq}. \\
For $\la = \La$, there exists a sequence of minimal solutions $\{ u _{\la_n}\}$ of the problem \eqref{eq:main_problem_ep}
 with $\la_n \uparrow \La$. By construction they are increasing with respect to $\la$. Let $\varphi_1$ be the solution of the problem \eqref{eq:P2}, then we have 
 \begin{equation}\label{eq:bound_minimal_seq}
 \begin{aligned}
 \mu \int_{\Omega} \frac{ u_{\la_n} \varphi_1}{|x|^2} \, dx 
+ \int_{\Omega} u_{\la_n}^{2^*-1} \varphi_1 \, dx 
+ \lambda_n \int_{\Omega} u_{\la_n}^{p-1} \varphi_1 \, dx = \la_{\varepsilon}^1 \int_{\om} u_{\la_n}\varphi_1  dx\leq \frac{1}{2} \int_{\Omega} u_{\la_n}^{2^*-1} \varphi_1 \, dx + C \int_{\om} \varphi_1 \, dx
\end{aligned}
 \end{equation}
 Let $\delta(x):= \text{dist}(x, \partial \Omega)$. Then by Lemma 2.2 of \cite{biswas2023boundary} there exists a bounded barrier function $\psi$ such that
\begin{equation*}
   \begin{cases}
       - \Delta \psi + \varepsilon (-\Delta)^s \psi \leq 0 &\text{ in } B_{4r} \setminus \bar{B}_r,\\
       0 \leq \psi \leq \tilde\ka r &\text{ in } B_r,\\
       \psi(x) \geq \frac{1}{\tilde \ka} \delta(x) & \text{ in } B_{4r} \setminus B_r,\\
       \psi \leq 0 & \text{ in } B_{4r}^C,
   \end{cases}
\end{equation*}
where $\tilde \ka$ is a constant independent of boundary of $B_{4r}$. 
 Using the comparison principle in \cite[Theorem 5.2]{biswas2025mixed} to compare $\varphi_1$ and $\psi$, we get ${\frac{\varphi_1}{\delta}(x_0)} \geq \frac{\psi}{\delta}(x_0)\geq \frac{1}{\tilde \ka}> 0$ at each $x_0 \in \partial \Omega$. But then Theorem 1.2 of \cite{biswas2023boundary} gives $\inf{\left( \frac{\varphi_1}{\delta}\right)} > 0$. So, using this and \eqref{eq:bound_minimal_seq}, we deduce that
 \begin{equation}\label{eq:min_seq_bd2}
 \begin{aligned}
     \inf{\left( \frac{\varphi_1}{\delta}\right)} &\left( \mu \int_{\Omega} \frac{ u_{\la_n} \delta}{|x|^2} \, dx + \int_{\Omega} u_{\la_n}^{2^*-1} \delta \, dx 
+ \lambda_n \int_{\Omega} u_{\la_n}^{p-1} \delta \, dx\right) \\
& \leq \mu \int_{\Omega} \frac{ u_{\la_n} \varphi_1}{|x|^2} \, dx 
+ \int_{\Omega} u_{\la_n}^{2^*-1} \varphi_1 \, dx 
+ \lambda_n \int_{\Omega} u_{\la_n}^{p-1} \varphi_1 \, dx \leq C .
\end{aligned}
 \end{equation}
 Let $\psi_1$ be the solution to the linear problem 
 \begin{equation*}
\begin{cases}
-\Delta \psi_1 + \varepsilon (-\Delta)^s \psi_1 = 1 & \text{in } \Omega, \\[8pt]
\psi_1 = 0 & \text{in } \mathbb{R}^N \setminus \Omega.
\end{cases}
\end{equation*}
Then taking $\psi_1$ as a test function in \eqref{eq:main_problem_ep} with $\la = \la_n$ and using Theorem $2.3$ of \cite{dhanya2026interior} and \eqref{eq:min_seq_bd2}, we get
\begin{align*}
    \int_{\om} u_{\la_n}dx & = \mu \int_{\Omega} \frac{ u_{\la_n} \psi_1}{|x|^2} \, dx 
+ \int_{\Omega} u_{\la_n}^{2^*-1} \psi_1 \, dx 
+ \lambda_n \int_{\Omega} u_{\la_n}^{p-1} \psi_1 \, dx \\
&  \leq C \left( \mu \int_{\Omega} \frac{ u_{\la_n} \delta}{|x|^2} \, dx + \int_{\Omega} u_{\la_n}^{2^*-1} \delta \, dx 
+ \lambda_n \int_{\Omega} u_{\la_n}^{p-1} \delta \, dx \right)\\
& \leq C.
\end{align*}
Hence the sequence $\{ u_{\la_n}\}$ is uniformly bounded in $L^1(\om)$ and converges to a function $u_{\La} \geq 0$. By monotone convergence theorem it follows that $u_{\La}$ is a weak solution.
 \end{proof}

\begin{lemma}
Let $\La$ be defined as in \eqref{eq:Lambda}. Then for $\la_0 \in (0,\La)$, either the problem \eqref{eq:main_problem_ep} has two distinct solutions or there exists a local minimum of functional $\Jepo$ in $\X(\om)$.
\end{lemma}
\begin{proof}
Let \(\lambda_0 \in (0,\Lambda)\) and fix \(\lambda_1 \in (\lambda_0, \Lambda)\) be fixed. Then by Lemma \ref{lem:existence_each_lambda} we get minimal solutions $u_{\lambda_0}$ and $u_{\la_1}$ of \eqref{eq:main_problem_ep} with $\lambda_0 $ and $\la_1$ respectively such that $u_{\lambda_0} \leq u_{\lambda_1}$. \\
Define the set $Z = \{ u \in \X(\Omega) : 0 \leq u \leq u_{\lambda_1} \}$. Then \(Z\) is a bounded, closed, and convex subset of \(\X(\Omega)\). Thus there exists $u_0 
\in Z$ such that
\begin{equation*}
\Jepo(u_0) = \inf_{u \in Z} \Jepo (u).
\end{equation*}
Since \(1 < p < 2\), we have $\Jepo(t u_{\lambda_1}) < 0$ for some sufficiently small $t > 0$. Moreover \(t u_{\lambda_1} \in Z\), it follows that \(u_0 \neq 0\).\\
If $u_0 \neq u_{\la_0}$, then we get two distinct solutions of \eqref{eq:main_problem_ep}. Otherwise, if $u_0 =  u_{\la_0}$, then we claim that \(u_0\) is a local minimum of the functional \(\Jepo\) in $\X(\om)$. We prove this by contradiction. Suppose, on the contrary, that \(u_0\) is not a local minimum in $\X(\om)$. Then there exists a sequence \(\{u_n\} \subset \X(\om)\) such that
\begin{equation*}
\lim_{n \to \infty} \|u_n - u_0\|_{\Xe} = 0
\quad \text{and} \quad
\Jepo(u_n) < \Jepo(u_0).
\end{equation*}
Now, consider the auxiliary functions $w_n \in \X(\om)$, and $v_n \in Z$ as
\begin{equation*}
w_n = (u_n - u_{\lambda_1})^+ \quad \text{and} \quad v_n(x) =
\begin{cases}
0, & u_n(x) \leq 0, \\[6pt]
u_n(x), & 0 \leq u_n(x) \leq u_{\lambda_1}(x), \\[6pt]
u_{\lambda_1}(x), & u_{\lambda_1}(x) \leq u_n(x).
\end{cases}
\end{equation*}
Also, take the sets
\begin{equation*}
A_n = \{ x : v_n(x) = u_n(x) \}, 
\qquad 
B_n = \{ x : u_n(x) \geq u_{\lambda_1}(x) \},
\qquad 
\widetilde{A}_n = A_n \cap \Omega, 
\qquad 
\widetilde{B}_n = B_n \cap \Omega.
\end{equation*}

Next, consider the function
\begin{equation*}
H_{\lambda_0}(t) = \frac{1}{2^*} t^{2^*}_+ + \frac{\lambda_0}{p} t^p_+.
\end{equation*}
Moreover, to simplify notation, we set
\begin{equation*}
U_n(x,y) = \frac{|u_n(x) - u_n(y)|^2}{|x-y|^{N+2s}},
\qquad
W_n(x,y) = \frac{|w_n(x) - w_n(y)|^2}{|x-y|^{N+2s}},
\end{equation*}
\begin{equation*}
U_n^{\pm}(x,y) = \frac{|(u_n)^\pm(x) - (u_n)^\pm(y)|^2}{|x-y|^{N+2s}},
\qquad V_n(x,y) = \frac{|v_n(x) - v_n(y)|^2}{|x-y|^{N+2s}}.
\end{equation*}
Then using $U_n(x,y) = U_n^+(x,y) + U_n^-(x,y) + 2 \frac{u_n^+(x) u_n^-(y)+ u_n^+(y) u_n^-(x)}{|x-y|^{N+2s}}$, we obtain
\begin{equation*}
\begin{aligned}
\Jepo(u_n) \;\geq\; & 
\frac{1}{2} \int_{\Omega} |\grad u_n^+|^2 \, dx 
+ \frac{1}{2} \int_{\Omega} |\nabla u_n^-|^2 \, dx
+ \frac{\varepsilon}{2} \iint_{\mathbb{R}^{2N}} U_n^{+}(x,y) \, dx \, dy  \\
& \quad + \frac{\varepsilon}{2} \iint_{\mathbb{R}^{2N}} U_n^-(x,y) \, dx \, dy
- \frac{\mu}{2} \int_{\Omega} \frac{(u_n^+)^2}{|x|^2} \, dx
- \int_{\widetilde{A}_n} H_{\lambda_0}(v_n) \, dx
- \int_{\widetilde{B}_n} H_{\lambda_0}(u_n) \, dx.\\
\geq&\; \frac{\varepsilon}{2}\,[u_n^- ]_{s}^2 
+ \Jepo(v_n) 
+ \frac{1}{2}\int_{\Omega} |\nabla u_n^-|^2 \, dx 
+ \frac{1}{2}\int_{\widetilde{B}_n} \big(|\nabla u_n^+|^2 - |\nabla v_n|^2\big) dx \\
&\quad + \frac{\varepsilon}{2} \iint_{B_n \times B_n} 
 \big[U_n^+(x,y)-V_n(x,y)\big] \, dx dy +  \varepsilon\int_{B_n} \int_{B_n^C} 
\big[U_n^+(x,y)-V_n(x,y)\big]  \, dy dx \\
&\quad - \frac{\mu}{2} \int_{\widetilde{B}_n} \frac{((u_n^+)^2 - v_n^2)}{|x|^2} \, dx
+ \int_{\widetilde{B}_n} \big( H_{\lambda_0}(v_n) - H_{\lambda_0}(u_n) \big) dx .
\end{aligned}
\end{equation*}
Since $w_n(x) = u_n(x) - u_{\lambda_1}(x)$ whenever $x \in B_n$, using $(37)-(40)$ of Lemma $7$ from \cite{shen_multiplicity_fractional} we obtain
\begin{equation*}
\begin{aligned}
\Jepo(u_n) 
&\;\geq\; \frac{\varepsilon}{2}[u_n^-]_s^2 
+ \frac{1}{2}\int_{\Omega} |\nabla u_n^-|^2 \, dx 
+ \Jepo(v_n) 
+ \frac{\varepsilon}{2}[w_n]_s^2 + \frac{1}{2}\int_{\widetilde{B}_n} |\nabla w_n|^2 \, dx\\
&\quad 
+  \int_{\widetilde{B}_n} \nabla w_n  \nabla u_{\lambda_1} \, dx + \varepsilon\iint_{B_n \times B_n} (w_n(x)-w_n(y))(u_{\lambda_1}(x)-u_{\lambda_1}(y)) \, d\nu \\
&\quad - 2\varepsilon \iint_{B_n \times B_n^c}
w_n(y)\big(u_n^+(x)-u_{\lambda_1}(y)\big) \, d\nu  - \frac{\mu}{2} \int_{\widetilde{B}_n} \frac{w_n^2(x) + 2 w_n(x) u_{\lambda_1}(x)}{|x|^2} \, dx \\
&\quad + \int_{\widetilde{B}_n} \big( H_{\lambda_0}(v_n) - H_{\lambda_0}(u_n) \big) dx \\
&\;\geq\; \frac{\varepsilon}{2}[u_n^-]_s^2 
+ \frac{1}{2}\int_{\Omega} |\nabla u_n^-|^2 \, dx 
+ \Jepo(v_n) 
+ \frac{\varepsilon}{2}[w_n]_s^2 + \frac{1}{2}\int_{\widetilde{B}_n} |\nabla w_n|^2 \, dx\\
&\quad 
+  \int_{\widetilde{B}_n} \nabla w_n  \nabla u_{\lambda_1} \, dx + \varepsilon \langle u_{\la_1}, w_n \rangle_s  - \frac{\mu}{2} \int_{\widetilde{B}_n} \frac{w_n^2(x) + 2 w_n(x) u_{\lambda_1}(x)}{|x|^2} \, dx \\
&\quad + \int_{\widetilde{B}_n} \big( H_{\lambda_0}(v_n) - H_{\lambda_0}(u_n) \big) dx,
\end{aligned}
\end{equation*}
where the last inequality follows from the fact that $Supp (w_n) = B_n$ and $u_n^+(x) \leq u_{\la_1}(x)$ for $x \in B_n^c$.
Since $u_{\lambda_1}$ is a supersolution of $(P_{\mu, \lambda_0}^{\varepsilon})$, we deduce
\begin{equation*}
\int_{\Omega} \nabla u_{\lambda_1}  \nabla w_n \, dx 
+ \varepsilon \langle u_{\lambda_1}, w_n \rangle_s 
\;\geq\; 
\mu \int_{\Omega} \frac{u_{\lambda_1} w_n}{|x|^2} \, dx 
+ \int_{\widetilde{B}_n} w_n H'_{\lambda_0}(u_{\lambda_1}) \, dx .
\end{equation*}
Thus we have
\begin{align}
\Jepo(u_n) \;\geq\;& \Jepo(v_n) + \frac{\varepsilon}{2} [w_n]_{s}^{2} 
+ \frac{1}{2} \int_{\widetilde{B}_n} |\nabla w_n|^2 
 +\int_{\widetilde{B}_n} \Bigg[ \frac{1}{2^*} \Big( u_{\lambda_1}^{2^*} - (u_{\lambda_1}+w_n)^{2^*} \Big) 
+ w_n \, u_{\lambda_1}^{2^*-1} \Bigg] dx \notag \\
& \quad - \frac{\mu}{2} \int_{\widetilde{B}_n} \frac{w_n^2(x)}{|x|^2} \, dx  + \lambda_0 \int_{\widetilde{B}_n} 
\frac{u_{\lambda_1}^p - (u_{\lambda_1}+w_n)^p}{p} 
+ w_n \, u_{\lambda_1}^{p-1} \, dx .\label{eq:J_lambda_expanded} 
\end{align}
Using Talyor's expansion and the fact that \( p \in (1,2) \), we obtain
\begin{equation*}
0 \;\leq\; \frac{1}{p} \Big[ (u_{\lambda_1}+w_n)^p - u_{\lambda_1}^p \Big] - w_n u_{\lambda_1}^{p-1} 
\;\leq\; \frac{(p-1)}{2} \, \frac{w_n^2}{u_{\lambda_1}^{2-p}}.
\end{equation*}
Moreover, using AM-GM inequality
\begin{equation*}
\big(u_{\lambda_1}(x) - u_{\lambda_1}(y)\big) 
\left( \frac{w_n^2(x)}{u_{\lambda_1}(x)} - \frac{w_n^2(y)}{u_{\lambda_1}(y)} \right) \leq |w_n(x) - w_n(y)|^2 .
\end{equation*}
Since \( u_{\lambda_1} \) is a solution of \eqref{eq:main_problem_ep} with $\la = \la_1$, using the above inequality together with Picone’s identity, we deduce
\begin{equation*}
\begin{aligned}
\lv w_n \rv_{\Xe}^2\;\geq\; \varepsilon \Big\langle u_{\lambda_1}, \frac{w_n^2}{u_{\lambda_1}} \Big\rangle_s 
+ \int_{\Omega} \nabla \!\left( \frac{w_n^2}{u_{\lambda_1}} \right) \nabla u_{\lambda_1} \, dx \;\geq\; \mu \int_{\Omega} \frac{w_n^2}{|x|^2} \, dx 
+ \lambda_0 \int_{\Omega} \frac{w_n^2}{u_{\lambda_1}^{2-p}} \, dx .
\end{aligned}
\end{equation*}
Thus, we obtain
\begin{equation}
\begin{aligned}
\lv w_n \rv_{\Xe}^2 - \mu \int_{\Omega} \frac{w_n^2}{|x|^2} \geq \lambda_0 \int_{\Omega} \frac{w_n^2}{u_{\lambda_1}^{2-p}}  \;\geq\; \frac{2 \la_0}{p-1} 
\int_{\widetilde{B}_n} 
\left[ \frac{(u_{\lambda_1}+w_n)^p - u_{\lambda_1}^p}{p} 
- w_n u_{\lambda_1}^{p-1} \right].
\end{aligned}
\label{eq:ineq46}
\end{equation}
Furthermore, we have the following inequality
\begin{equation}
0 \;\leq\; \frac{1}{2^*} \Big[ (u_{\lambda_1}+w_n)^{2^*} - u_{\lambda_1}^{2^*} \Big] - w_n u_{\lambda_1}^{2^*-1}
\;\leq\; C \left( u_{\lambda_1}^{2^*-2} w_n^2 + w_n^{2^*} \right).
\label{eq:ineq47}
\end{equation}
Plugging \eqref{eq:ineq46} and \eqref{eq:ineq47} into \eqref{eq:J_lambda_expanded}, we deduce
\begin{equation*}
\begin{aligned}
\Jepo(u_n) 
&\geq\Jepo(v_n) 
+ \frac{2-p}{2} 
\left[ \varepsilon [w_n]_s^2 + \int_{\widetilde{B}_n} |\nabla w_n|^2  
- \mu \int_{\widetilde{B}_n} \frac{w_n^2(x)}{|x|^2} \right] - C \int_{\Omega} \Big( u_{\lambda_1}^{2^*-2} w_n^2 + w_n^{2^*} \Big)\\
& \geq \Jepo(v_n) + C \|w_n\|_{\Xe}^2 
- C \int_{\widetilde{B}_n} \Big( u_{\lambda_1}^{2^*-2} w_n^2 + w_n^{2^*} \Big) dx .
\end{aligned}
\end{equation*}
Moreover, using the idea of Lemma $7$ of \cite{shen_multiplicity_fractional} and $u_0 = u_{\la_0}$, we obtain
\begin{equation*}
\lim_{n \to \infty} |\widetilde{B}_n| = 0 \quad \text{and} \quad \int_{\widetilde{B}_n} \Big( u_{\lambda_1}^{2^*-2} w_n^2 + w_n^{2^*} \Big) dx \leq o_n(1) \lv w_n\rv_{\Xe}^2.
\end{equation*}
Hence, we finally conclude
\begin{equation*}
\begin{aligned}
\Jepo(u_n) 
&\;\geq\; \Jepo(v_n) + \big(C - o_n(1)\big)\|w_n\|_{\Xe}^2 \geq  \Jepo(u_0),
\end{aligned}
\end{equation*}
for large $n$. This gives us the contradiction. Thus, $u_0$ is a local minimum for $\Jepo$.
\end{proof}
Next, we examine the existence of a second positive solution of the form
\begin{equation*}
u = u_{\lambda} + v.
\end{equation*}
The equation for $v$ becomes
\begin{equation}
\left\{
\begin{aligned}
- \Delta v + \varepsilon  (-\Delta)^s v - \mu \frac{v}{|x|^2} &= \lambda (u_{\lambda} + v)^{p-1} + (u_{\lambda} + v)^{2^* - 1} - \lambda u_{\lambda}^{p-1} - u_{\lambda}^{2^* - 1}, && \text{in } \Omega, \\
v &= 0, && \text{in } \mathbb{R}^N \setminus \Omega,
\end{aligned}
\right.
\label{eq:second_solution_equation}
\end{equation}
Define the function \( g(x,t) \) as
\begin{equation*}
g(x,t) = 
\begin{cases}
(u_{\lambda} + t)^{2^*-1} - u_{\lambda}^{2^*-1} + \lambda \left[ (u_{\lambda} + t)^{p-1} - u_{\lambda}^{p-1} \right], & \text{if } t > 0, \\
0, & \text{if } t \leq 0,
\end{cases}
\end{equation*}
and its primitive
\begin{equation*}
G(v) = \int_0^v g(x,t) \, dt.
\end{equation*}
The associated energy functional becomes
\begin{equation*}
\Ie(v) = \frac{1}{2} \left[ \int_{\Omega} |\nabla v|^2 + \varepsilon  [v]_{s}^2 - \mu \int_{\Omega} \frac{v^2}{|x|^2} \right] - \int_{\Omega} G(v) \, dx.
\end{equation*}
There is a one-to-one correspondence between the critical points of $\Ie$ in $\X$ and the weak solutions of problem \eqref{eq:second_solution_equation}.
We now prove the existence of a nontrivial solution \( v \) using a contradiction argument. Assume that \( v = 0 \) is the only solution of \eqref{eq:second_solution_equation}.

\begin{lemma}
The functional \( \Ie \) has a local minimum at $v =0$ in \( \X(\om) \).
\end{lemma}
\begin{proof}
Let \( w \in \X (\om)\). Then using the fact that \( \langle (\Je)^{\prime}(u_{\lambda}), w^+ \rangle = 0 \) and rearranging terms, we deduce
\begin{equation*}
\begin{aligned}
\Ie(w) &= \frac{1}{2} \int_{\Omega} |\nabla w|^2 - \mu \int_{\Omega} \frac{w^2}{|x|^2} \, dx + \frac{\varepsilon }{2} [w]^2_s - \frac{1}{2^*} \int_{\Omega} \left[ (u_{\lambda} + w^+)^{2^*} - u_{\lambda}^{2^*} - 2^* u_{\lambda}^{2^* - 1} w^+ \right] \\
&\quad - \frac{\lambda}{ p} \int_{\Omega} \left[ (u_{\lambda} + w^+)^{p} - u_{\lambda}^{p} - p u_{\lambda}^{p-1} w^+ \right] \\
&  = \Je(u_{\lambda} + w^+) - \Je(u_{\lambda})  + \frac{1}{2} \int_{\Omega} |\nabla w^-|^2 - \mu \int_{\Omega} \frac{(w^-)^2}{|x|^2} + \frac{\varepsilon }{2} [w^-]_s^2 + \varepsilon  \langle w^+, w^- \rangle_s.
\end{aligned}
\end{equation*}
Since \( u_{\lambda} \) is a local minimizer and
\begin{equation*}
\frac{1}{2} [w^-]_s^2 + \langle w^+, w^- \rangle_s = \frac{1}{2} [w]_s^2 - \frac{1}{2} [w^+]_s^2 \geq 0,
\end{equation*}
we obtain the inequality
\begin{equation*}
\Ie(w) \geq \frac{1}{2} \left( \int_{\Omega} |\nabla w^-|^2 - \mu \int_{\Omega} \frac{(w^-)^2}{|x|^2} \right),
\end{equation*}
for \( \|w\|_{\X, \varepsilon} \) small enough. This concludes the proof.
\end{proof}

\begin{lemma}
If \( v = 0 \) is the only critical point of  \( \Ie \), then \( \Ie \) satisfies the Palais–Smale condition (PS)\(_c\) for any \( c < \frac{1}{N} S_{\mu}^{N/2} \).
\end{lemma}
\begin{proof}
Let \( \{v_n\} \subset \X(\om) \) be a (PS)\(_c\) sequence. Then
\begin{equation*}
\Ie(v_n) \to c < \frac{1}{N} S_{\mu}^{N/2}, \quad  (\Ie)^{\prime}(v_n) \to 0 \quad \text{in } (\X)'.
\end{equation*}
Then using Hardy inequality and H\"{o}lder inequality, we compute
\begin{equation*}
\begin{aligned}
c + o_n(1)& \|u_{\lambda} + v_n\|_{\X, \varepsilon}  \geq 2^* \Ie(v_n) - \langle (\Ie)^{\prime}(v_n),u_{\la}+v_n\rangle\\
&= \left( \frac{2^*}{2} - 1 \right) \left[ \int_{\Omega} |\nabla v_n|^2 - \mu \int_{\Omega} \frac{v_n^2}{|x|^2} + \varepsilon [v_n]^2_s \right] + (2^* - 2) \int_{\Omega} u_{\lambda}^{2^* - 1} v_n^+   \\
&\quad +\int_{\Omega} u_{\lambda}^{2^* - 1} v_n^-+ \la\left( \frac{2^*}{ p} - 1 \right) \int_{\om} u_{\la}^{p} + \la(2^*-2) \int_{\Omega} u_{\lambda}^{p-1} v_n^+ \\
& \quad+ \lambda \left( \frac{p - 2^*}{ p} \right) \int_{\Omega} (u_{\lambda} + v_n^+)^{p} + \lambda \int_{\Omega} v_n^{p-1} v_n^-.\\
&\geq \left( \frac{2}{N - 2} \right) \left( 1 - \frac{\mu}{\bar{\mu}} \right) \left( \int_{\Omega} |\nabla v_n|^2 + \varepsilon  [v_n]_s^2 \right) 
+ C \lambda \frac{(p  - 2^*)}{ p} \|u_{\lambda} + v_n^+\|_{\X}^{p }.
\end{aligned}
\end{equation*}
Since \( 1 < p < 2 \), this implies that \( \{v_n\} \) is bounded in \( \X \). Thus, passing to a subsequence if necessary, we have 
\begin{equation}\label{convergence_properties}
\begin{aligned}
v_n &\rightharpoonup v_0 \quad \quad \;\text{in } \X,\\
v_n(x) &\to v_0(x) \quad \text{a.e. in } \Omega,\\
v_n &\to v_0 \quad \quad \;\text{in } L^t(\Omega) \text{ for all } 1 < t < 2^*.
\end{aligned}
\end{equation}
Clearly, \( v_0 \) is a critical point of \( \Ie \). But our assumption implies \( v_0 = 0 \).\\
Using the Brézis–Lieb Lemma and \eqref{convergence_properties}, we deduce
\begin{equation}\label{eq:reduced_I_functional}
\Ie(v_n) = \frac{1}{2} \int_{\Omega} |\nabla v_n|^2 - \frac{\mu}{2} \int_{\Omega} \frac{v_n^2}{|x|^2} + \frac{\varepsilon }{2} [v_n]_s^2 - \frac{1}{2^*} \int_{\Omega} (v_n^+)^{2^*} + o_n(1).
\end{equation}
Also, we compute
\begin{equation*} 
\left\langle (\Ie)^{\prime}(v_n), u_\la +v_n \right\rangle = \int_{\Omega} |\nabla v_n|^2 + \varepsilon  [v_n]_s^2 - \mu \int_{\Omega} \frac{v_n^2}{|x|^2} + \int_{\Omega} (v_n^+)^{2^*} + o_n(1).
\end{equation*}
Now, we can assume that
\begin{equation*}
\int_{\Omega} |\nabla v_n|^2 + \varepsilon  [v_n]_s^2 - \mu \int_{\Omega} \frac{v_n^2}{|x|^2} \to b,
\end{equation*}
and
\begin{equation*}
\int_{\Omega} (v_n^+)^{2^*} \to b, \quad \text{as } n \to \infty.
\end{equation*}

If \( b = 0 \), the proof is complete. If \( b \neq 0 \), then by the definition of the Sobolev constant \( S_{\mu} \), we have
\begin{equation*} 
b \geq S_{\mu} b^{2 / 2^*},
\end{equation*}
i.e., $b \geq S_{\mu}^{N/2}$. Thus, using \eqref{eq:reduced_I_functional} we obtain the inequality
\begin{equation*}
c = o_n(1) + \Ie(v_n) = \frac{b}{N} + o_n(1) \geq \frac{S_{\mu}^{N/2}}{N} + o_n(1),
\end{equation*}
which contradicts our earlier assumption on the energy level \( c \).
\end{proof}

\begin{proof}[Proof of Theorem \ref{thm:sublinear_two_solutions}]
We define the min–max value as
\begin{equation*}
c^*_{\lambda} = \inf_{h \in \Gamma} \max_{t \in [0,1]} \Ie(h(t)),
\end{equation*}
where $\Gamma = \left\{ h \in C([0,1], \X(\om)) : h(0) = 0,\, h(1) = t_0 u_{\varepsilon,\al} \right\}$, and \( t_0 \) is chosen, independent of $\varepsilon$, such that $\Ie(t_0 u_{\varepsilon,\al}) < 0$.

We aim to show the existence of a nontrivial critical point of \( \Ie \) using the Mountain Pass Theorem (see \cite[Theorem $1$]{ghoussoub_mountain_pass_theorem}). To accomplish this, it remains only to verify that
\begin{equation*}
c^*_{\lambda} < \frac{1}{N} S_{\mu}^{N/2}.
\end{equation*}
Using the standard inequality $(b + d)^n \geq b^n + d^n + n b^{n-1} d$, for $n > 1; \; b,d > 0$, 
we deduce that
\begin{equation*}
g(x, a) \geq a^p + p  u_\lambda^{p-1} a.
\end{equation*}
Choosing \( \delta \) small enough such that
\[
u_{\lambda} \geq M_0 > 0 \quad \text{in } B_{2 \de} \setminus \{0\}
\]
for some $M_0>0$ independent of $\varepsilon$. Then using the functions $u_{\varepsilon,\al}$ defined in subsection \ref{sec:superlinear}, we obtain
\begin{equation*} 
\begin{aligned}
\Ie(t u_{\varepsilon,\al}) &\leq \frac{t^2}{2} \left( \int_{\Omega} |\nabla u_{\varepsilon,\al}|^2 - \mu \int_{\Omega} \frac{u_{\varepsilon,\al}^2}{|x|^2} + \varepsilon[u_{\varepsilon,\al}]_s^2 - M_0^{2^*-2}(2^*-1) \int_{\Omega} u_{\varepsilon,\al}^2 \, dx \right) 
- \frac{t^{2^*}}{2^*} \int_{\Omega} u_{\varepsilon,\al}^{2^*} \, dx\\
& = \frac{t^2}{2}A - \frac{t^{2^*}}{2^*}B\\
& \leq \frac{1}{N}   \frac{A^{2^* / (2^* - 2)}}{B^{2 / (2^* - 2)}}
\end{aligned}
\end{equation*}
Using \eqref{eq:norm_lower_bound}, \eqref{eq:norm_upper_bound}, \eqref{eq:fractional_norm_bound}, and \eqref{eq:lower_bound_norm_r} with $r = 2$, we have
\begin{equation*} 
\begin{aligned}
\Ie(t u_{\varepsilon,\al}) &\leq \frac{1}{N} \frac{\left( S_{\mu}^{N/2} + C \varepsilon^{\al (N - 2)} + C \varepsilon^{1+2 \al (1-s)\sqrt{\bar{\mu}}/\sqrt{\bar{\mu} - \mu}} - 
 C \varepsilon^{2 \al \sqrt{\bar{\mu}}/\sqrt{\bar{\mu} - \mu}}  \right)^{2^*/(2^*-2)}}{\left(S_{\mu}^{N/2}-C \varepsilon^{\al N}\right)^{2/(2^*-2)}}\\
& = \frac{1}{N} S_{\mu}^{N/2}  \frac{(1 - C\varepsilon^{2 \al \sqrt{\bar{\mu}}/\sqrt{\bar{\mu}-\mu}})^{N/2}}{(1 - C \varepsilon^{\al N})^{(N - 2)/2}} < \frac{1}{N} S_{\mu}^{N/2}.
\end{aligned}
\end{equation*}
provided that \( \mu < \bar{\mu} - 1 \) and we choose $\al$ such that $\min\{\al(N-2), 1+2 \al (1-s)\sqrt{\bar{\mu}}/\sqrt{\bar{\mu} - \mu}\} > 2 \al \sqrt{\bar{\mu}}/\sqrt{\bar{\mu} - \mu}$.
Therefore, we conclude
\begin{equation*} 
c^*_{\lambda} \leq \max_{t > 0} \Ie(t u_{\varepsilon,\al}) < \frac{1}{N} S_{\mu}^{N/2}. \qedhere
\end{equation*}
\end{proof}
\vspace{.2cm}
{\textbf{Acknowledgment:}}\\
The author, Shammi Malhotra, is supported by the Prime Minister’s Research Fellowship (PMRF ID - 1403226). The author, Sarika Goyal, would like to thank the Anusandhan National Research Foundation, Department of Science and Technology, Government of India for the financial support under the grant SPG/2022/002068.

\bibliographystyle{abbrv}
\bibliography{references}

\begin{thebibliography}{10}

\bibitem{peral2016CZ}
B.~Abdellaoui, M.~Medina, I.~Peral, and A.~Primo.
\newblock The effect of the {H}ardy potential in some {C}alder\'on-{Z}ygmund properties for the fractional {L}aplacian.
\newblock {\em J. Differential Equations}, 260(11):8160--8206, 2016.

\bibitem{adimurthi_hardy_inequality}
Adimurthi.
\newblock Hardy-{S}obolev inequality in {$H^1(\Omega)$} and its applications.
\newblock {\em Commun. Contemp. Math.}, 4(3):409--434, 2002.

\bibitem{alvino2017radial_polya}
A.~Alvino, F.~Brock, F.~Chiacchio, A.~Mercaldo, and M.~R. Posteraro.
\newblock Some isoperimetric inequalities on {$\mathbb{R}^N$} with respect to weights {$|x|^\alpha$}.
\newblock {\em J. Math. Anal. Appl.}, 451(1):280--318, 2017.

\bibitem{abc_laplacian}
A.~Ambrosetti, H.~Brezis, and G.~Cerami.
\newblock Combined effects of concave and convex nonlinearities in some elliptic problems.
\newblock {\em J. Funct. Anal.}, 122(2):519--543, 1994.

\bibitem{biagi2022brezis}
S.~Biagi, S.~Dipierro, E.~Valdinoci, and E.~Vecchi.
\newblock A {B}rezis-{N}irenberg type result for mixed local and nonlocal operators.
\newblock {\em NoDEA Nonlinear Differential Equations Appl.}, 32(4):Paper No. 62, 28, 2025.

\bibitem{biagi2024mixed_hardy}
S.~Biagi, F.~Esposito, L.~Montoro, and E.~Vecchi.
\newblock On mixed local–nonlocal problems with {H}ardy potential.
\newblock {\em Proceedings of the Royal Society of Edinburgh: Section A Mathematics}, page 1–34, 2025.

\bibitem{biswas2025mixed}
A.~Biswas and M.~Modasiya.
\newblock Mixed local-nonlocal operators: maximum principles, eigenvalue problems and their applications.
\newblock {\em Journal d'Analyse Math{\'e}matique}, pages 1--35, 2025.

\bibitem{biswas2023boundary}
A.~Biswas, M.~Modasiya, and A.~Sen.
\newblock Boundary regularity of mixed local-nonlocal operators and its application.
\newblock {\em Ann. Mat. Pura Appl. (4)}, 202(2):679--710, 2023.

\bibitem{brasco_second_eigenvalue}
L.~Brasco and E.~Parini.
\newblock The second eigenvalue of the fractional {$p$}-{L}aplacian.
\newblock {\em Adv. Calc. Var.}, 9(4):323--355, 2016.

\bibitem{brezis_1983}
H.~Br\'ezis and L.~Nirenberg.
\newblock Positive solutions of nonlinear elliptic equations involving critical {S}obolev exponents.
\newblock {\em Comm. Pure Appl. Math.}, 36(4):437--477, 1983.

\bibitem{caffarelli1984interpolation_sobolev}
L.~Caffarelli, R.~Kohn, and L.~Nirenberg.
\newblock First order interpolation inequalities with weights.
\newblock {\em Compositio Math.}, 53(3):259--275, 1984.

\bibitem{cao_peng2003signchanging}
D.~Cao and S.~Peng.
\newblock A note on the sign-changing solutions to elliptic problems with critical {S}obolev and {H}ardy terms.
\newblock {\em J. Differential Equations}, 193(2):424--434, 2003.

\bibitem{chen2003estimates}
J.~Chen.
\newblock Existence of solutions for a nonlinear {PDE} with an inverse square potential.
\newblock {\em J. Differential Equations}, 195(2):497--519, 2003.

\bibitem{chen2004multiple}
J.~Chen.
\newblock Multiple positive solutions for a class of nonlinear elliptic equations.
\newblock {\em J. Math. Anal. Appl.}, 295(2):341--354, 2004.

\bibitem{squassina_regularity}
W.~Chen, S.~Mosconi, and M.~Squassina.
\newblock Nonlocal problems with critical {H}ardy nonlinearity.
\newblock {\em J. Funct. Anal.}, 275(11):3065--3114, 2018.

\bibitem{cotsiolis2004best_fractional}
A.~Cotsiolis and N.~K. Tavoularis.
\newblock Best constants for {S}obolev inequalities for higher order fractional derivatives.
\newblock {\em J. Math. Anal. Appl.}, 295(1):225--236, 2004.

\bibitem{silva2024mixed}
J.~a.~V. da~Silva, A.~Fiscella, and V.~A.~B. Viloria.
\newblock Mixed local-nonlocal quasilinear problems with critical nonlinearities.
\newblock {\em J. Differential Equations}, 408:494--536, 2024.

\bibitem{dhanya2026interior}
R.~Dhanya, J.~Giacomoni, and R.~Jana.
\newblock Interior and boundary regularity of mixed local nonlocal problem with singular data and its applications.
\newblock {\em Nonlinear Anal.}, 262:Paper No. 113940, 2026.

\bibitem{castro2014nonlocal_harnack}
A.~Di~Castro, T.~Kuusi, and G.~Palatucci.
\newblock Nonlocal {H}arnack inequalities.
\newblock {\em J. Funct. Anal.}, 267(6):1807--1836, 2014.

\bibitem{palatucci_castro2016local}
A.~Di~Castro, T.~Kuusi, and G.~Palatucci.
\newblock Local behavior of fractional {$p$}-minimizers.
\newblock {\em Ann. Inst. H. Poincar\'e{} C Anal. Non Lin\'eaire}, 33(5):1279--1299, 2016.

\bibitem{dipierro2016fractionalwithHardy}
S.~Dipierro, L.~Montoro, I.~Peral, and B.~Sciunzi.
\newblock Qualitative properties of positive solutions to nonlocal critical problems involving the {H}ardy-{L}eray potential.
\newblock {\em Calc. Var. Partial Differential Equations}, 55(4):Art. 99, 29, 2016.

\bibitem{valdinoci_mixed_application}
S.~Dipierro and E.~Valdinoci.
\newblock Description of an ecological niche for a mixed local/nonlocal dispersal: an evolution equation and a new {N}eumann condition arising from the superposition of {B}rownian and {L}\'evy processes.
\newblock {\em Phys. A}, 575:Paper No. 126052, 20, 2021.

\bibitem{ferrero2001existence}
A.~Ferrero and F.~Gazzola.
\newblock Existence of solutions for singular critical growth semilinear elliptic equations.
\newblock {\em J. Differential Equations}, 177(2):494--522, 2001.

\bibitem{frank2008hardy}
R.~L. Frank, E.~H. Lieb, and R.~Seiringer.
\newblock Hardy-{L}ieb-{T}hirring inequalities for fractional {S}chr\"odinger operators.
\newblock {\em J. Amer. Math. Soc.}, 21(4):925--950, 2008.

\bibitem{frank_hardy_applications}
W.~M. Frank, D.~J. Land, and R.~M. Spector.
\newblock Singular potentials.
\newblock {\em Rev. Modern Phys.}, 43(1):36--98, 1971.

\bibitem{garain_higher_holder_regularity}
P.~Garain and J.~Kinnunen.
\newblock On the regularity theory for mixed local and nonlocal quasilinear elliptic equations.
\newblock {\em Trans. Amer. Math. Soc.}, 375(8):5393--5423, 2022.

\bibitem{ghoussoub_mountain_pass_theorem}
N.~Ghoussoub and D.~Preiss.
\newblock A general mountain pass principle for locating and classifying critical points.
\newblock {\em Ann. Inst. H. Poincar\'e{} C Anal. Non Lin\'eaire}, 6(5):321--330, 1989.

\bibitem{ghoussoub_fractional_hardy}
N.~Ghoussoub, F.~Robert, S.~Shakerian, and M.~Zhao.
\newblock Mass and asymptotics associated to fractional {H}ardy-{S}chr\"odinger operators in critical regimes.
\newblock {\em Comm. Partial Differential Equations}, 43(6):859--892, 2018.

\bibitem{giusti2003iteration_lemma}
E.~Giusti.
\newblock {\em Direct methods in the calculus of variations}.
\newblock World Scientific Publishing Co., Inc., River Edge, NJ, 2003.

\bibitem{pigong_local_p_hardy}
P.~Han.
\newblock Quasilinear elliptic problems with critical exponents and {H}ardy terms.
\newblock {\em Nonlinear Anal.}, 61(5):735--758, 2005.

\bibitem{heinonen_potential_theory}
J.~Heinonen, T.~Kilpel\"ainen, and O.~Martio.
\newblock {\em Nonlinear potential theory of degenerate elliptic equations}.
\newblock Oxford Mathematical Monographs. The Clarendon Press, Oxford University Press, New York, 1993.
\newblock Oxford Science Publications.

\bibitem{jannelli_hardy_starting}
E.~Jannelli.
\newblock The role played by space dimension in elliptic critical problems.
\newblock {\em J. Differential Equations}, 156(2):407--426, 1999.

\bibitem{giovanni_fractional_book}
G.~Leoni.
\newblock {\em A first course in fractional {S}obolev spaces}, volume 229 of {\em Graduate Studies in Mathematics}.
\newblock American Mathematical Society, Providence, RI, 2023.

\bibitem{liebloss2001analysis}
E.~H. Lieb and M.~Loss.
\newblock {\em Analysis}, volume~14 of {\em Graduate Studies in Mathematics}.
\newblock American Mathematical Society, Providence, RI, second edition, 2001.

\bibitem{malhotra2025eigenvalues}
S.~Malhotra, S.~Goyal, and K.~Sreenadh.
\newblock On the eigenvalues and {F}u{\v{c}}{\'\i}k spectrum of $p$-laplace local and nonlocal operator with mixed interpolated {H}ardy term.
\newblock {\em Asymptotic Analysis}, page 09217134251339280, 2025.

\bibitem{olivia_classfication_localp}
F.~Oliva, B.~Sciunzi, and G.~Vaira.
\newblock Radial symmetry for a quasilinear elliptic equation with a critical {S}obolev growth and {H}ardy potential.
\newblock {\em J. Math. Pures Appl. (9)}, 140:89--109, 2020.

\bibitem{peral_hardy_applications}
I.~P.~Alonso and F.~S.~de Diego.
\newblock {\em Elliptic and parabolic equations involving the {H}ardy-{L}eray potential}, volume~38 of {\em De Gruyter Series in Nonlinear Analysis and Applications}.
\newblock De Gruyter, Berlin, 2021.

\bibitem{shang_zhang_fractional_laplacian}
X.~Shang, J.~Zhang, and R.~Yin.
\newblock Existence of positive solutions to fractional elliptic problems with {H}ardy potential and critical growth.
\newblock {\em Math. Methods Appl. Sci.}, 42(1):115--136, 2019.

\bibitem{shen_multiplicity_fractional}
Y.~Shen.
\newblock Multiplicity of positive solutions to a critical fractional equation with {H}ardy potential and concave-convex nonlinearities.
\newblock {\em Complex Var. Elliptic Equ.}, 67(9):2152--2180, 2022.

\bibitem{stein_weiss1957fractional}
E.~M. Stein and G.~Weiss.
\newblock Fractional integrals on {$n$}-dimensional {E}uclidean space.
\newblock {\em J. Math. Mech.}, 7:503--514, 1958.

\bibitem{terracini_minimizers}
S.~Terracini.
\newblock On positive entire solutions to a class of equations with a singular coefficient and critical exponent.
\newblock {\em Adv. Differential Equations}, 1(2):241--264, 1996.

\bibitem{willem2012minimax}
M.~Willem.
\newblock {\em Minimax theorems}, volume~24 of {\em Progress in Nonlinear Differential Equations and their Applications}.
\newblock Birkh\"auser Boston, Inc., Boston, MA, 1996.

\end{thebibliography}
\end{document}